\documentclass{amsart}
\usepackage{latexsym,amsxtra,amscd,ifthen}
\usepackage{amsfonts}
\usepackage{verbatim}
\usepackage{amsmath}
\usepackage{amsthm}
\usepackage{amssymb}

\usepackage{rotfloat}

\numberwithin{equation}{section}

\theoremstyle{plain}
\newtheorem{theorem}{Theorem}[section]
\newtheorem{lemma}[theorem]{Lemma}
\newtheorem{proposition}[theorem]{Proposition}
\newtheorem{corollary}[theorem]{Corollary}

\theoremstyle{definition}
\newtheorem{definition}[theorem]{Definition}

\newtheorem*{remark*}{Remark}

\newtheorem{subsec}[theorem]{}

\newcommand{\calE}{{\mathcal E}}
\newcommand{\calO}{{\mathcal O}}
\newcommand{\CC}{{\mathbb C}}
\newcommand{\NN}{{\mathbb N}}
\newcommand{\QQ}{{\mathbb Q}}
\newcommand{\ZZ}{{\mathbb Z}}
\newcommand{\Xl}{\underline{X}}
\newcommand{\ml}{\underline{m}}

\newcommand{\itil}{\widetilde{i}}
\newcommand{\jtil}{\widetilde{j}}
\newcommand{\ltil}{\widetilde{l}}
\newcommand{\mtil}{\widetilde{m}}
\newcommand{\stil}{\widetilde{s}}
\newcommand{\ttil}{\widetilde{t}}
\newcommand{\Itil}{\widetilde{I}}
\newcommand{\Jtil}{\widetilde{J}}

\newcommand{\qhat}{\widehat{q}}

\newcommand{\Oq}{\calO_q}
\newcommand{\OqMth}{\Oq(M_3(k))}
\newcommand{\OqMn}{\Oq(M_n(k))}
\newcommand{\OqGLth}{\Oq(GL_3(k))}
\newcommand{\OqGLn}{\Oq(GL_n(k))}
\newcommand{\OqSLth}{\Oq(SL_3(k))}
\newcommand{\OqSLn}{\Oq(SL_n(k))}
\newcommand{\kx}{k^\times}

\newcommand{\gnoc}{\mathrel{{\lower.2ex\hbox{$\backsim$}}\llap{\raise.45ex\hbox{=}}}}

\newcommand{\lowunder}[1]{$\underline{\vphantom{p} \hbox{#1}}$}

\newcommand{\Hspec}{H\text{-}\operatorname{spec}}
\newcommand{\SHspec}{SH\text{-}\operatorname{spec}}

\DeclareMathOperator\spec{\operatorname{spec}}
\DeclareMathOperator\prim{\operatorname{prim}}
\DeclareMathOperator\Fract{\operatorname{Fract}}
\DeclareMathOperator\GKdim{\operatorname{GK{.}dim}}
\DeclareMathOperator\calExt{\operatorname{Ext}}
\DeclareMathOperator\gldim{\operatorname{gl{.}dim}}

\input xy \xyoption{matrix} \xyoption{arrow}
\newcommand{\edge}{\ar@{-}}
\newcommand{\place}{*{}+0}
\newcommand{\plb}{*{\bullet}+0}
\newcommand{\plc}{*{\circ}+0}

\newcommand{\ccc}{\plc&\plc&\plc}
\newcommand{\bcc}{\plb&\plc&\plc}

\newcommand{\ccb}{\plc&\plc&\plb}
\newcommand{\bbc}{\plb&\plb&\plc}
\newcommand{\bcb}{\plb&\plc&\plb}
\newcommand{\cbb}{\plc&\plb&\plb}

\newcommand{\hrz}{\place \edge[r] &\place}

\newcommand{\hrzvrt}{\place \edge[r] \edge[d] &\place \edge[d]}

\newcommand{\dropvarup}[2]{\save+<0ex,#1ex>\drop{#2}\restore}
\newcommand{\dropup}[1]{\save+<0ex,4ex>\drop{#1}\restore}

\newcommand{\dropgen}[3]{\save+<#1ex,#2ex>\drop{#3}\restore}

\begin{document}

\title[Primitive ideals in quantum $SL_3$ and $GL_3$]
{Primitive ideals in quantum $SL_3$ and $GL_3$}

\author{K.R. Goodearl and T.H. Lenagan}

\address{Goodearl: Department of Mathematics, 
University of California at Santa Barbara,
Santa Barbara, CA 93106, USA}

\email{goodearl@math.ucsb.edu} 

\address{Lenagan: Maxwell Institute for Mathematical Sciences, School of Mathematics, University of Edinburgh, JCMB, King's Buildings, Mayfield Road, Edinburgh EH9 3JZ, Scotland, UK}

\email{tom@maths.ed.ac.uk}

\begin{abstract}
Explicit generating sets are found for all primitive ideals in
the generic quantized coordinate rings of $SL_3$ and $GL_3$ over an arbitrary
algebraically closed field $k$. (Previously, generators were only known up to certain localizations.) These generating sets form polynormal
regular sequences, from which it follows that all primitive factor
algebras of $\OqSLth$ and $\OqGLth$ are Auslander-Gorenstein and
Cohen-Macaulay. 
\end{abstract}



\thanks{This research was partially supported by grants from the NSF
(USA) and by Leverhulme Research Interchange Grant F/00158/X (UK)}

\maketitle

\setcounter{section}{-1}
\section{Introduction} \label{intro}

The primitive ideals of quantum $SL_3$ were first classified by Hodges and Levasseur in the case of $\Oq(SL_3(\CC))$ \cite[Theorems 4.2.2, 4.3.1, 4.4.1 and \S4.5]{HL}. (Here and throughout, we consider only generic quantized coordinate rings, meaning that quantizing parameters such as $q$ are not roots of unity.) This classification was extended to $\Oq(SL_n(\CC))$ in \cite{HLn}, to $\Oq(G)$ for semisimple  groups $G$ and $q$ transcendental in \cite{Jos}, and to multiparameter quantizations $\calO_{q,p}(G)$ over $\CC$ in \cite{HLT}. In these classifications, the primitive ideals appear as pullbacks of maximal ideals from certain localizations, and it is only in the localizations that generating sets are calculated (assuming the base field is algebraically closed). The only case in which generating sets for primitive ideals have been explicitly determined is the easy case of $\Oq(SL_2(\CC))$ \cite[Theorem B.1.1]{HL}. Of course, once the primitive (or prime) ideals of $\OqSLn$ are known in some form, those for $\OqGLn$ can be readily obtained via the isomorphism $\OqGLn\cong \OqSLn[z^{\pm1}]$ observed by Levasseur and Stafford \cite[Proposition]{LS}. 

Our purpose here is to calculate explicit generating sets for all primitive ideals of the (generic) quantized coordinate rings $\OqSLth$ and $\OqGLth$, over any algebraically closed base field $k$. With some care in the choices of generators and the order in which they are listed, we actually obtain generating sets which are polynormal regular sequences, from which we obtain strong homological conclusions: All primitive factor algebras of $\OqSLth$ and $\OqGLth$ are Auslander-Gorenstein and
Cohen-Macaulay (with respect to Gelfand-Kirillov dimension). Further, addressing a question raised in \cite{GZ}, we show that all maximal ideals of $\OqSLth$ and $\OqGLth$ have codimension $1$.

The line of our calculations follows the framework given by stratifications relative to torus actions (see \cite[Theorems II.2.13 and II.8.4]{BG}), which we recall below. It is convenient to work first with $\OqGLth$, since the results for $\OqSLth$ are then immediate corollaries. On the other hand, with appropriate (obvious) modifications, all steps of our calculations can be performed in parallel in $\OqSLth$.

\begin{subsec} \label{stratification} 
{\bf Stratification.} For the remainder of the introduction, set $A:= \OqGLth$ and $H:= (\kx)^6$. There is a standard rational action of $H$ on $A$ by $k$-algebra automorphisms (see \S\ref{subsec1.5}), and we write $\Hspec A$ to denote the set of $H$-prime ideals of $A$. (These coincide with the $H$-stable prime ideals of $A$ by \cite[Proposition II.2.9]{BG}.) The prime and primitive spectra of $A$ are stratified via $H$ as in \cite[Definition II.2.1]{BG}:
$$\spec A = \bigsqcup_{J\in\Hspec A} \spec_JA  \quad\qquad\text{and}\quad\qquad \prim A = \bigsqcup_{J\in\Hspec A} \prim_JA,$$
where the \emph{$H$-strata} $\spec_JA$ and $\prim_JA$ are the sets
\begin{align*}
\spec_J A &:= \{P\in\spec A \mid \bigcap_{h\in H} h(P)=J\}  \\
\prim_J A &:=(\prim A) \cap (\spec_J A).
\end{align*}

Given $J\in \Hspec A$, the strata $\spec_J A$ and $\prim_J A$ have the following structure \cite[Theorem II.2.13, Corollaries II.6.5 and II.8.5]{BG}.
\begin{enumerate}
\item{\it The set $\calE_J$ of all regular $H$-eigenvectors in $A/J$ is a denominator set, and the localization $A_J := (A/J)[\calE_J^{-1}]$ is $H$-simple (with respect to the induced $H$-action).}
\item{\it The center $Z(A_J)$ is a commutative Laurent polynomial ring over the field $k$, in at most $6$ indeterminates.}
\item{\it Localization, contraction and extension provide bijections
\begin{align*}
\spec_J A &\longleftrightarrow \spec A_J \longleftrightarrow \spec Z(A_J) \\
\prim_J A & \longleftrightarrow \max A_J \longleftrightarrow \max Z(A_J).
\end{align*}
In particular, $\prim_J A$ is the set of maximal elements of $\spec_J A$.}
\end{enumerate}

There are exactly analogous results for the algebra $\OqSLth$, relative to a standard action of a torus of rank $5$ (see \S\ref{subsec1.5}).
\end{subsec}

Our route to compute $\prim A$ follows the structure outlined above. We give more detail for the individual steps below.

\begin{subsec} \label{Hspec}
{\bf $H$-prime ideals and generating sets.} The first step is to determine $\Hspec A$. It is known that $A$ has exactly $36$ $H$-prime ideals, induced from those of $\OqMth$ which do not contain the quantum determinant. Explicit generating sets for the $H$-primes of $\OqMth$ were calculated in \cite{GLenJAlg}, and from these we obtain explicit generating sets for the $H$-primes of $A$ (see \S\ref{subsec2.1}). All these generators are quantum minors.

The stratification of $\spec A$ corresponds to an analogous one for $\spec\OqSLth$, which coincides with the partition introduced by Hodges and Levasseur in \cite{HL}. Their partition was indexed by the double Weyl group $S_3\times S_3$, and we carry over their indexing to write 
$$\Hspec A= \{Q_w \mid w\in S_3\times S_3\}$$
(see \S\ref{subsec2.2}). It is convenient to use similar indexing for the $H$-strata of $\spec A$ and $\prim A$. Namely, for $w=(w_+,w_-)$ in $S_3\times S_3$, we set
\begin{align*}
\spec_wA &= \spec_{w_+,w_-} A := \spec_{Q_w} A \\
\prim_wA &= \prim_{w_+,w_-} A := \prim_{Q_w} A.
\end{align*}
\end{subsec}

\begin{subsec} \label{redloc}
{\bf Reduced localizations.} For computational purposes, it is helpful to reduce the localizations $A_J$ by shrinking the denominator sets $\calE_J$ to multiplicative sets consisting of normal elements, provided the reduced localization remains $H$-simple. (Recall that an algebra with an $H$-action is called {\it $H$-simple\/} provided it is nonzero and it has no proper nonzero $H$-stable ideals. These conditions hold, in particular, if the algebra is prime and has no nonzero $H$-prime ideals.) We identify appropriate denominator sets $E_w \subseteq \calE_{Q_w}$ and set $A_w := (A/Q_w)[E_w^{-1}]$ for $w\in S_3\times S_3$ (\S\S\ref{subsec3.2}-\ref{subsec3.3}). The localization $A_w$ satisfies the same properties as $A_{Q_w}$ (\S\ref{subsec3.3}). Namely,
\begin{enumerate}
\item{\it $A_w$ is $H$-simple.}
\item{\it $Z(A_w)$ is a Laurent polynomial ring over $k$ in at most $6$ indeterminates.}
\item{\it Localization, contraction and extension provide bijections}
\begin{align*}
\spec_w A &\longleftrightarrow \spec A_w \longleftrightarrow \spec Z(A_w)  \\
\prim_w A &\longleftrightarrow \max A_w \longleftrightarrow \max Z(A_w).
\end{align*}
\end{enumerate}
\end{subsec}

\begin{subsec} \label{indets}
{\bf Indeterminates.} The next step is to calculate each of the rings $Z(A_w)$, as a Laurent polynomial ring in specific indeterminates. We do this in Lemma \ref{lem4.2}. (The bound of $6$ on the number of indeterminates is not sharp -- as it turns out, each $Z(A_w)$ has Krull dimension at most $3$.)

Once we have expressed $Z(A_w)$ in the form $k[z_1^{\pm1},\dots,z_d^{\pm1}]$, we can identify the primitive ideals in the stratum $\prim_w A$ via \S\ref{redloc}(3), assuming $k$ is algebraically closed. They are exactly the inverse images of the ideals
\begin{equation} \label{locprimitive}
(A/Q_w) \cap \bigl( A_w(z_1-\alpha_1) +\cdots+ A_w(z_d-\alpha_d) \bigr),
\tag{E0.4}  
\end{equation}
for $\alpha_1,\dots,\alpha_d\in \kx$, under the quotient map $A\rightarrow A/Q_w$. However, this description only provides generators up to localization. Hence, one step remains.
\end{subsec} 

\begin{subsec} \label{primgens}
{\bf Primitive generators.}  For $w\in S_3\times S_3$ and $\alpha_1,\dots,\alpha_d\in \kx$, we find elements $a_1,\dots,a_d\in A/Q_w$ which generate a prime ideal of $A/Q_w$, and which generate the same ideal of $A_w$ as $z_1-\alpha_1,\dots,z_d-\alpha_d$ (Lemma \ref{lem5.3} and proof of Theorem \ref{mainthm}). It follows that $a_1,\dots,a_d$ generate the ideal described in \eqref{locprimitive}. Combining coset representatives of the $a_i$ with generators for $Q_w$, finally, we obtain sets of generators for the primitive ideals in $\spec_w A$ (Theorem \ref{mainthm}). 
\end{subsec}


\section{Background and notation} \label{backnot}

Fix a base field $k$ throughout, and a nonzero element $q\in k$ which is  not a root of unity. In our main theorem, we require $k$ to be algebraically closed, but that assumption is not needed for most of the prior results. For this section, also fix an integer $n\ge 2$; later, we specialize to the case $n=3$.

\begin{subsec} \label{subsec1.1} {\bf Generators and relations.} For $n\in\NN$, we present the quantized coordinate ring of the matrix variety $M_n(k)$ as the $k$-algebra $\OqMn$ with generators $X_{ij}$ for $1\le i,j\le n$ and relations
\begin{equation}
\begin{aligned}
 X_{ij}X_{lj} &= qX_{lj}X_{ij}  &\qquad\qquad&(i<l)  \\
X_{ij}X_{im} &= qX_{im}X_{ij}  &&(j<m)  \\
X_{ij}X_{lm} &= X_{lm}X_{ij}  &&(i<l,\; j>m)  \\
X_{ij}X_{lm} - X_{lm}X_{ij} &= \qhat X_{im} X_{lj}  &&(i<l,\; j<m) \,, 
\end{aligned}  \tag{E1.1}
\end{equation}
where $\qhat := q-q^{-1}$.
The {\it quantum determinant\/} in $\OqMn$ is the element
$$D_q := \sum_{\pi\in S_n} (-q)^{\ell(\pi)} X_{1,\pi(1)} X_{2,\pi(2)}
\cdots X_{n,\pi(n)} \,,$$
where $\ell(\pi)$ denotes the {\it length\/} of the permutation
$\pi$, that is, the minimum length of an expression for $\pi$ as a
product of simple transpositions $(i,i+1)$. The element $D_q$ lies in the center of $\OqMn$ (e.g., \cite[Theorem 4.6.1]{PW}). The quantized coordinate rings of $GL_n(k)$ and $SL_n(k)$ are then given as a localization and a quotient of $\OqMn$, respectively:
\begin{align*}
\OqGLn &:= \OqMn[D_q^{-1}] ; &\OqSLn &:=    \OqMn/\langle D_q-1\rangle.
\end{align*}
Let us use $x_{ij}$ to denote the coset of $X_{ij}$ in $\OqSLn$.

Due to the assumption that $q$ is not a root of unity,
\begin{enumerate}
\item {\it All prime ideals of $\OqMn$, $\OqGLn$, and $\OqSLn$ are completely prime}
\end{enumerate}
\cite[Corollary II.6.10]{BG}, meaning that quotients modulo these prime ideals are domains.
\end{subsec}

\begin{subsec} \label{subsec1.2} {\bf Quantum minors.} For any nonempty sets $I,J\subseteq \{1,\dots,n\}$ of the same cardinality, we write $[I|J]$ for the {\it quantum minor\/} with row index set $I$ and column index set $J$ in $\OqMn$, that is, for  the quantum determinant of $\Oq(M_{|I|}(k)$ in the subalgebra $k\langle X_{ij}\mid i\in I,\; j\in J\rangle \subseteq \OqMn$. In detail, if $I = \{i_1<\cdots<i_t\}$ and $J= \{j_1<\cdots<j_t\}$, then
$$[I|J] := \sum_{\pi\in S_t} (-q)^{\ell(\pi)} X_{i_1,j_{\pi(1)}} X_{i_2,j_{\pi(2)}} \cdots X_{i_t,j_{\pi(t)}}\,.$$
The corresponding quantum minor in $\OqSLn$ is obtained by replacing $X_{i,\pi(i)}$ by $x_{i,\pi(i)}$ in the formula above. It is just the coset of $[I|J]$, and we will use the same notation for it. 

We often omit set braces and commas from the notation for quantum minors, writing $[1|3]$ and  $[12|23]$ for $[\{1\}|\{3\}]= X_{13}$ and $[\{1,2\}|\{2,3\}]$, for instance. Complementary index sets will appear in several formulas; we will use the notation 
$$\Itil:= \{1,\dots,n\} \setminus I.$$
Abbreviated set notation will be used here too, as in $\widetilde{23}= \widetilde{\{2,3\}}$.
\end{subsec}

\begin{subsec} \label{subsec1.3} {\bf Quantum Laplace and commutation relations.} Some relations among  quantum minors are needed for our computations; we record them here. We state ones holding in $\OqMn$ (and thus also in $\OqGLn$), and simply note that analogous ones hold in $\OqSLn$.

The {\it quantum Laplace relations\/} say that
\begin{align}
\sum_{j=1}^n (-q)^{j-l} X_{ij}[\ltil|\jtil] &= \delta_{il}D_q  &\sum_{j=1}^n (-q)^{i-j} [\jtil|\itil]X_{jl} &= \delta_{il}D_q  \tag{E1.3a}
\end{align}
for all $i$, $l$ \cite[Corollary 4.4.4]{PW}. There are many commutation relations among quantum minors. Here are seven of the ones from \cite[Lemmas 4.5.1, 5.1.2]{PW}:
\begin{gather*}
\begin{align}
\qquad\qquad X_{ij}[\ltil|\mtil] &= [\ltil|\mtil]X_{ij}  &&\qquad(l\ne i,\; m\ne j)  \tag{E1.3b}
\end{align}\\
\begin{aligned}
X_{ij}[\ltil|\jtil]- q[\ltil|\jtil]X_{ij} &= \qhat\sum_{s<j} (-q)^{s-j}[\ltil|\stil]X_{is}  \\
X_{ij}[\ltil|\jtil]- q^{-1}[\ltil|\jtil]X_{ij} &= -\qhat\sum_{s>j} (-q)^{s-j}[\ltil|\stil]X_{is} \\
\end{aligned}  \tag{E1.3c} \qquad\qquad (l\ne i)\\
\begin{aligned}
X_{ij}[\itil|\mtil]- q[\itil|\mtil]X_{ij} &= \qhat\sum_{s<i} (-q)^{s-i}[\stil|\mtil]X_{sj}  \\
X_{ij}[\itil|\mtil]- q^{-1}[\itil|\mtil]X_{ij} &= -\qhat\sum_{s>i} (-q)^{s-i}[\stil|\mtil]X_{sj}
\end{aligned}  \tag{E1.3d}  \qquad\qquad(m\ne j)\\
\begin{aligned}
X_{ij}[\itil|\jtil]- [\itil|\jtil]X_{ij} &= q\qhat \biggl( \sum_{s<i} (-q)^{s-i} X_{sj}[\stil|\jtil] -\sum_{t>j} (-q)^{j-t} [\itil|\ttil]X_{it} \biggr)  \\
X_{ij}[\itil|\jtil]- [\itil|\jtil]X_{ij} &= q^{-1}\qhat \biggl( \sum_{t<j} (-q)^{j-t} [\itil|\ttil]X_{it} - \sum_{s>i} (-q)^{s-i} X_{sj}[\stil|\jtil] \biggr) \,.  
\end{aligned}  \tag{E1.3e}
\end{gather*}
Finally, we give some commutation relations among $(n-1)\times(n-1)$ quantum minors \cite[Theorem 5.2.1]{PW}:
\begin{align}
[\itil|\jtil][\itil|\mtil] &= q^{-1}[\itil|\mtil][\itil|\jtil]  &&(j<m)  \tag{E1.3f}\\
[\itil|\jtil][\ltil|\jtil] &= q^{-1}[\ltil|\jtil][\itil|\jtil]  &&(i<l)  \tag{E1.3g}\\
[\itil|\jtil][\ltil|\mtil] &= [\ltil|\mtil][\itil|\jtil]  &&(i<l,\; j>m)  \tag{E1.3h}\\
[\itil|\jtil][\ltil|\mtil] -[\ltil|\mtil][\itil|\jtil] &= \qhat [\itil|\mtil][\ltil|\jtil]  &&(i<l,\; j<m) \,.  \tag{E1.3i}
\end{align}
\end{subsec}

\begin{subsec} \label{subsec1.4} {\bf Symmetry.} The algebras $\OqMn$, $\OqGLn$, and $\OqSLn$ enjoy a number of symmetries, in the form of automorphisms and anti-automorphisms. We single out three. First, there is the {\it transpose automorphism\/} $\tau$ on $\OqMn$, which is a $k$-algebra automorphism such that $\tau(X_{ij})= X_{ji}$ for all $i$, $j$ \cite[Proposition 3.7.1(1)]{PW}. This automorphism also transposes rows and columns in quantum minors:
$$\tau\bigl( [I|J] \bigr)= [J|I]$$
for all $I$, $J$ \cite[Lemma 4.3.1]{PW}. In particular, $\tau(D_q)=D_q$, and so $\tau$ induces automorphisms of $\OqGLn$ and $\OqSLn$, which we also denote $\tau$.

Let $S$ denote the antipode of the Hopf algebra $\OqGLn$; this is a $k$-algebra anti-automorphism such that $S(X_{ij})= (-q)^{i-j} [\jtil|\itil]D_q^{-1}$ for all $i$, $j$ (see \cite[Theorem 5.3.2]{PW}, with $q$ and $q^{-1}$ interchanged). The action of $S$ on quantum minors is given by
$$S\bigl( [I|J] \bigr)= (-q)^{\Sigma I- \Sigma J} [\Jtil|\Itil] D_q^{-1}$$
for all $I$, $J$ \cite[Lemma 4.1]{KLR}. Since the antipode of $\OqSLn$, which we also denote by $S$, is induced by the antipode of $\OqGLn$, we have $S(x_{ij})= (-q)^{i-j} [\jtil|\itil]$ for all $i$, $j$, and the displayed formula becomes $S\bigl( [I|J] \bigr)= (-q)^{\Sigma I- \Sigma J} [\Jtil|\Itil]$ in $\OqSLn$.

From \cite[Proposition 3.7.1(3)]{PW}, there is a $k$-algebra anti-automorphism $\rho$ of $\OqMn$ sending each $X_{ij}$ to $X_{n+1-j,n+1-i}$. This sends $D_q$ to itself \cite[Lemma 4.2.3]{PW}, so $\rho$ induces anti-automorphisms of $\OqGLn$ and $\OqSLn$. Further, by \cite[Lemma 4.3.1]{PW}, 
$$\rho\bigl( [I|J] \bigr)= [w_0(J)|w_0(I)]$$
for all quantum minors $[I|J]$ in $\OqMn$, where $w_0 := \left[ \begin{smallmatrix} 1&2&\cdots&n\\ n&n-1&\cdots&1 \end{smallmatrix} \right]$ is the longest element of the symmetric group $S_n$.
\end{subsec}

\begin{subsec} \label{subsec1.5} {\bf Torus actions.} Write $H$ (or $H_n$, if it is necessary to specify $n$) for the algebraic torus $(\kx)^{2n}$, and let $H$ act on $\OqMn$ and $\OqGLn$ in the standard way, namely by $k$-algebra automorphisms such that
$$(\alpha_1,\dots,\alpha_n,\beta_1,\dots,\beta_n).X_{ij}= \alpha_i\beta_j X_{ij}$$
for all $i$, $j$. Then define the subgroup
$$SH := \{(\alpha_1,\dots,\alpha_n,\beta_1,\dots,\beta_n) \in H \mid \alpha_1\alpha_2\cdots\alpha_n\beta_1\beta_2\cdots \beta_n=1 \}$$
of $H$. Since $SH$ fixes $D_q-1$, it induces an action (by $k$-algebra automorphisms) on $\OqSLn$. The actions of $H$ on $\OqMn$ and $\OqGLn$, and the action of $SH$ on $\OqSLn$, are rational. All quantum minors in $\OqMn$ (respectively, $\OqSLn$) are $H$-eigenvectors (respectively, $SH$-eigenvectors).

Although the transpose automorphism $\tau$ of $\OqMn$ is not $H$-equivariant, it does satisfy
$$\tau \bigl( (\alpha_1,\dots,\alpha_n,\beta_1,\dots,\beta_n).Y \bigr)= (\beta_1,\dots,\beta_n,\alpha_1,\dots,\alpha_n).\tau(Y)$$
for all $(\alpha_1,\dots,\alpha_n,\beta_1,\dots,\beta_n) \in H$ and $Y\in \OqMn$. (To see this, just check the displayed identity for $Y=X_{ij}$.) Consequently, $\tau$ maps $H$-stable subsets of $\OqMn$ to $H$-stable subsets, and likewise in $\OqGLn$. Similarly, $\tau$ maps $SH$-stable subsets of $\OqSLn$ to $SH$-stable subsets.

As for the $S$ and $\rho$, we have
\begin{align*}
S \bigl( (\alpha_1,\dots,\alpha_n,\beta_1,\dots,\beta_n).Y \bigr) &= (\beta_1^{-1},\dots,\beta_n^{-1},\alpha_1^{-1},\dots,\alpha_n^{-1}).S(Y)  \\
\rho\bigl( (\alpha_1,\dots,\alpha_n,\beta_1,\dots,\beta_n).Y \bigr) &= (\beta_n,\dots,\beta_1,\alpha_n,\dots,\alpha_1).\rho(Y)
\end{align*}
for all $(\alpha_1,\dots,\alpha_n,\beta_1,\dots,\beta_n) \in H$ and $Y\in \OqGLn$. Consequently, $S$ and $\rho$ send $H$-stable subsets of $\OqGLn$ to $H$-stable subsets, and similarly, $SH$-stable subsets of $\OqSLn$ are mapped to $SH$-stable subsets.

The sets of torus-invariant prime ideals in the algebras under discussion will be denoted
\begin{align*}
\Hspec \OqMn; &&\Hspec \OqGLn; &&\SHspec \OqSLn.
\end{align*}
Note that all three collections are stable under $\tau$ and $\rho$, and that the latter two are stable under $S$.
Recall from \cite[Proposition II.2.9]{BG} that the $H$-prime ideals of $\OqMn$ or $\OqGLn$ coincide with the prime $H$-stable ideals, and that the $SH$-prime ideals of $\OqSLn$ coincide with the prime $SH$-stable ideals. 
\end{subsec}


\section{Torus-invariant prime ideals} \label{Hprimes}

For most of this section, we concentrate on the case $n=3$ and take $A := \OqGLth$.

\begin{subsec} \label{subsec2.1} {\bf Generators.} Generating sets for all $H$-prime ideals in $\OqMth$ were determined in \cite{GLenJAlg}. For the $36$ $H$-primes not containing the quantum determinant, these generating sets are encoded in \cite[Figure 6]{GLenJAlg}, a version of which we give in Figure \ref{Hprimegens} below. Since all prime ideals of $A$ are induced from prime ideals of $\OqMth$, Figure \ref{Hprimegens} also gives generating sets for all the $H$-prime ideals of $A$.

The positions in Figure \ref{Hprimegens} are indexed by pairs $(w_+,w_-)$ of permutations from $S_3$; we will explain this indexing in \S\ref{subsec2.2} below. Each of the small $3\times 3$ diagrams within the figure represents a selection of $1\times 1$ and $2\times 2$ quantum minors, given in positions relative to the $3\times 3$ matrix $(X_{ij})$ of canonical generators for $\OqMth$. Circles $(\circ)$ are placeholders, bullets $(\bullet)$ stand for $1\times 1$ quantum minors (i.e., generators $X_{ij}$), and squares $(\square)$ stand for $2\times 2$ quantum minors. For example, the diagram in position $(321,321)$ records an empty set of generators, while the diagram in position $(231,231)$ records the generators $[12|23]$ and $X_{31}$.
\end{subsec}

\begin{figure} [h]
 $$\xymatrixrowsep{0.1pc}\xymatrixcolsep{0.1pc}
\xymatrix{
 &&&&\ar@{=}[28,0] \\
 &&&\dropgen{-0.5}0{w_-} \\  
 &&& &&&\dropvarup0{321} &&&&\dropvarup0{231} &&&&\dropvarup0{312} &&&&\dropvarup0{132} &&&&\dropvarup0{213} &&&&\dropvarup0{123} \\
 &\dropvarup0{w_+} \\
\ar@{=}[0,28] &&&&& &&&&& &&&&& &&&&& &&&&& &&&\dropup{} \\
 &&& &&\ccc &&\plc &\hrzvrt &&\ccb &&\cbb &&\ccb &&\cbb \\
 &&\dropvarup0{321} & &&\ccc &&\plc &\hrz &&\ccc &&\ccc
&&\ccb &&\ccb \\
 &&& &&\ccc &&\ccc &&\ccc &&\ccc &&\ccc &&\ccc \\   \\
 &&& &&\ccc &&\plc &\hrzvrt &&\ccb &&\cbb &&\ccb &&\cbb \\
 &&\dropvarup0{231} & &&\ccc &&\plc &\hrz &&\ccc &&\ccc
&&\ccb &&\ccb \\
 &&& &&\bcc &&\bcc &&\bcc &&\bcc &&\bcc &&\bcc \\   \\
 &&& &&\ccc &&\plc &\place \edge[dd] \edge[r] &\place \edge[d] &&\ccb &&\cbb &&\ccb &&\cbb \\
 &&\dropvarup0{312} & &&\hrzvrt &\plc &&\place \edge[rr] \edge[d] &&\place
 &&\hrzvrt &\plc &&\hrzvrt &\plc &&\hrzvrt &\plb &&\hrzvrt &\plb \\
 &&& &&\hrz &\plc &&\hrz &\plc &&\hrz &\plc &&\hrz &\plc
 &&\hrz &\plc &&\hrz &\plc \\   \\
 &&& &&\ccc &&\plc &\hrzvrt &&\ccb &&\cbb &&\ccb &&\cbb \\
 &&\dropvarup0{132} & &&\bcc &&\plb &\hrz &&\bcc &&\bcc
&&\bcb &&\bcb \\
 &&& &&\bcc &&\bcc &&\bcc &&\bcc &&\bcc &&\bcc \\   \\
 &&& &&\ccc &&\plc &\hrzvrt &&\ccb &&\cbb &&\ccb &&\cbb \\
 &&\dropvarup0{213} & &&\ccc &&\plc &\hrz &&\ccc &&\ccc
&&\ccb &&\ccb \\
 &&& &&\bbc &&\bbc &&\bbc &&\bbc &&\bbc &&\bbc \\   \\
 &&& &&\ccc &&\plc &\hrzvrt &&\ccb &&\cbb &&\ccb &&\cbb \\
 &&\dropvarup0{123} & &&\bcc &&\plb &\hrz &&\bcc &&\bcc
&&\bcb &&\bcb \\
 &&& &&\bbc &&\bbc &&\bbc &&\bbc &&\bbc &&\bbc \\   
 &&&&&
}$$
\caption{Generators for $H$-prime ideals}   \label{Hprimegens}
\end{figure}

\begin{subsec} \label{subsec2.2} {\bf Indexing.} The indexing in Figure \ref{Hprimegens} is adapted from that used by Hodges and Levasseur in \cite[\S2.2]{HL} for certain key ideals $I_w= I_{w_+,w_-}$ generated by quantum minors in $\OqSLth$. In fact, one can show that the $I_w$ are exactly the $SH$-prime ideals of $\OqSLth$. Some of the quantum minors that belong to $I_w$ do not appear among the generators given in \cite{HL}, but we find it useful to include them. Thus, the sets of quantum minors appearing in Figure \ref{Hprimegens} are slightly larger than the sets used to define the ideals $I_w$. To give an explicit formula for our generating sets, it is convenient to recall the following relations among index sets.

Let $I$ and $J$ be finite sets of positive integers with the same cardinality, say $t$. Write the elements of these sets in ascending order:
\begin{align*}
I &= \{i_1<\cdots<i_t\}  &J &= \{j_1<\cdots<j_t\}.
\end{align*}
Then define $I\le J$ if and only if $i_l\le j_l$ for all $l=1,\dots,t$. 

For $y\in S_3$, define the following ideals of $A$:
\begin{align*}
Q^+_y &:= \bigl\langle [i|1] \bigm| i\nleq y(1)\bigr\rangle +\bigl\langle [i|2] \bigm| i\nleq y(1),y(2)\bigr\rangle + \bigl\langle [I|12] \bigm| I\nleq y(12) \bigr\rangle \\
Q^-_y &:= \bigl\langle [\itil|\widetilde{1}] \bigm| i\nleq y(1) \bigr\rangle+ \bigl\langle [\Itil|\widetilde{12}] \bigm| I \nleq y(12) \bigr\rangle+ \bigl\langle [\Itil|\widetilde{13}] \bigm| I \nleq y(12),y(13) \bigr\rangle,
\end{align*}
where, following our previous conventions, $y(12)= y(\{1,2\})$.
Finally, for $w=(w_+,w_-)$ in $S_3\times S_3$, set
$$Q_w= Q_{w_+,w_-} := Q^+_{w_+}+ Q^-_{w_-}\,.$$
The diagram in position $w$ of Figure \ref{Hprimegens} gives a set of quantum minors that generate $Q_w$, where we label a permutation $y$ by its effect on $1,2,3$. In other words, if $y$ is given in input-output form by the matrix $\left[\smallmatrix 1&2&3\\ y(1)&y(2)&y(3) \endsmallmatrix\right]$, we just record the bottom row of this matrix. Note that generating sets for the ideals $Q^+_{w_+}= Q_{w_+,321}$ are given in the left column of Figure \ref{Hprimegens}, and those for $Q^-_{w_-}= Q_{321,w_-}$ in the first row.
\end{subsec}

\begin{subsec} \label{symmHprimes} {\bf Symmetry.}
By inspection of the generating sets in Figure \ref{Hprimegens}, we see that
\begin{align*}
\tau\bigl( Q^+_y \bigr) &= Q^-_{y^{-1}}  &\tau\bigl( Q^-_y \bigr) &= Q^+_{y^{-1}}  
\end{align*}
for all $y\in S_3$. Consequently,
\begin{equation}
\tau(Q_{w_+,w_-})= Q_{w_-^{-1},w_+^{-1}}
\tag{E2.3a}
\end{equation}
for $w_+,w_-\in S_3$, and therefore $\tau$ induces isomorphisms
\begin{equation} \label{tauquo}
A/Q_{w_+,w_-} \;{\overset\tau\longrightarrow}\; A/Q_{w_-^{-1},w_+^{-1}}\,.  \tag{E2.3b}
\end{equation}

Further inspection reveals relations such as $S(Q^+_{132}) \subseteq Q^+_{132}$ in $A$. Since $S$ sends $H$-primes of $A$ to $H$-primes and preserves proper inclusions, we obtain
\begin{align*}
S\bigl( Q^+_y \bigr) &= Q^+_{y^{-1}}  &S\bigl( Q^-_y \bigr) &= Q^-_{y^{-1}}
\end{align*}
for $y\in S_3$. Thus,
\begin{equation}
S(Q_{w_+,w_-})= Q_{w_+^{-1},w_-^{-1}}    \tag{E2.3c}
\end{equation}
for $w_+,w_- \in S_3$, and therefore $S$ induces anti-isomorphisms
\begin{equation} \label{Squo}
A/Q_{w_+,w_-} \ {\overset{S}\longrightarrow}\  A/Q_{w_+^{-1},w_-^{-1}}\,.    \tag{E2.3d}
\end{equation}

Finally, we find that
\begin{align*}
\rho(Q^+_y) &= Q^+_{w_0y^{-1}w_0}  &\rho(Q^-_y) &= Q^-_{w_0y^{-1}w_0}
\end{align*}
for $y\in S_3$, and thus
\begin{equation}
\rho(Q_{w_+,w_-})= Q_{w_0w_+^{-1}w_0,w_0w_-^{-1}w_0}    \tag{E2.3e}
\end{equation}
for $w_+,w_- \in S_3$. Hence, $\rho$ and $\rho\tau$ induce anti-isomorphisms 
\begin{equation} \label{rhoquo}
\begin{aligned}
A/Q_{w_+,w_-} \ &{\overset{\rho}\longrightarrow}\  A/Q_{w_0w_+^{-1}w_0,w_0w_-^{-1}w_0}   \\
A/Q_{w_+,w_-} \ &{\overset{\rho\tau}\longrightarrow}\  A/Q_{w_0w_-w_0,w_0w_+w_0}  
\end{aligned}  \tag{E2.3f}
\end{equation}
for all $w_+,w_- \in S_3$.

For convenience in identifying the (anti-) isomorphisms \eqref{tauquo}, \eqref{Squo}, \eqref{rhoquo}, we list the permutations $y^{-1}$, $w_0y^{-1}w_0$, and $w_0yw_0$, for $y\in S_3$, in Figure \ref{useful}.
\begin{figure}[h]
$$\begin{matrix}
y  &\qquad &y^{-1} &\qquad &w_0y^{-1}w_0 &\qquad &w_0yw_0 \\  \\
321&&321&&321&&321\\
231&&312&&231&&312\\
312&&231&&312&&231\\
132&&132&&213&&213\\
213&&213&&132&&132\\
123&&123&&123&&123\\
\end{matrix}$$
\caption{}   \label{useful}
\end{figure}  
\end{subsec}

\begin{subsec} \label{subsec2.3} {\bf Normal elements.} The relations (E1.1) and (E1.3b--e) imply that various quantum minors are normal in $\OqMth$ and $A$, or become normal modulo certain ideals. In particular:
\begin{enumerate}
\item {\it $X_{13}$, $X_{31}$, $[12|23]$, and $[23|12]$ are normal.
\item $X_{12}$ and $X_{23}$ are normal modulo $\langle X_{13}\rangle$.
\item $X_{21}$ and $X_{32}$ are normal modulo $\langle X_{31}\rangle$.}
\end{enumerate}
Inspection of Figure \ref{Hprimegens} immediately reveals that for each $H$-prime ideal $Q$ of $A$, the given generators can be listed in a sequence $a_1,\dots,a_t$ such that $a_1$ is normal and $a_i$, for $i>1$, is normal modulo $\langle a_1,\dots,a_{i-1}\rangle$. Thus, we have a {\it polynormal sequence\/} of generators. Further, Figure \ref{Hprimegens} shows that each of the ideals $\langle a_1,\dots,a_{i-1}\rangle$ is ($H$-) prime, and so $a_i$ is regular modulo $\langle a_1,\dots,a_{i-1}\rangle$. Hence, our list of generators is also a {\it regular sequence\/}. To summarize:
\begin{enumerate}
\item[(4)] {\it Each $H$-prime ideal of $A$ has a polynormal regular sequence of generators.}
\end{enumerate}
\end{subsec}

We now turn to the $SH$-prime ideals of $\OqSLth$, and show that they are exactly the push-forwards of the $H$-prime ideals of $\OqGLth$ with respect to the quotient map. This holds for $\OqSLn$ for arbitrary $n$, and we record the result in that generality. Let us write $\pi: \OqMn\rightarrow \OqSLn$ for the quotient map and also for the natural extension of this map to $\OqGLn$.

The following proposition, in the case when $k=\CC$ and $q$ is transcendental over $\QQ$, is a corollary of \cite[Lemme 3.4.10]{Lau}.

\begin{proposition} \label{prop2.4} The set map $P\mapsto \pi(P)$ provides a bijection of $\Hspec \OqGLn$ onto $\SHspec \OqSLn$.
\end{proposition}

\begin{proof} Set $A= \OqGLn$ and $B=\OqSLn$, and let $B[z^{\pm1}]$ be a Laurent polynomial ring over $B$. By \cite[Proposition]{LS}, there is a $k$-algebra isomorphism $\xi:A\rightarrow B[z^{\pm1}]$ such that $\xi(X_{1j})= zx_{1j}$ for all $j$ while $\xi(X_{ij})= x_{ij}$ for all $i\ge2$ and all $j$. As shown in \cite[Lemma II.5.16]{BG}, there is a bijection $\SHspec B \rightarrow \Hspec A$ given by the rule $Q\mapsto \xi^{-1}(Q[z^{\pm1}])$. Hence, we need only show that $\xi(P)= \pi(P)[z^{\pm1}]$ for all $P\in \Hspec A$.

Any $P\in \Hspec A$ is generated by $P\cap \OqMn$, so it is generated by the $H$-eigen\-vectors in $P\cap \OqMn$. Thus, it will suffice to show that for any $H$-eigenvector $a\in \OqMn$, there is a unit $u\in B[z^{\pm1}]$ such that $\xi(a)= u\pi(a)$.

Recall that $\OqMn$ has a $K$-basis consisting of the lexicographically ordered monomials in the generators $X_{ij}$. Let us write such monomials in the form 
$$\Xl^{\ml}= X_{11}^{m_{11}} X_{12}^{m_{12}} \cdots X_{nn}^{m_{nn}},$$
where $\ml= (m_{11},m_{12},\dots,m_{nn}) \in \ZZ_{\ge0}^{n^2}$. Consider an $H$-eigenvector 
$$a= \sum_{s=1}^t \lambda_s \Xl^{\ml_s} \in \OqMn,$$
where the $\ml_s$ are distinct elements of $\ZZ_{\ge0}^{n^2}$ and $\lambda_s\in\kx$. If $h=(q,1,1,\dots,1)\in H$, then
\begin{align*}
h.a &= \sum_{s=1}^t q^{r_s} \lambda_s \Xl^{\ml_s} \in \OqMn,
&r_s &= \sum_{j=1}^n (\ml_s)_{1j}.
\end{align*}
Since $a$ is an $H$-eigenvector and $q$ is not a root of unity, $r_s=r_1$ for all $s$. Consequently,
$$\xi(a)= \sum_{s=1}^t z^{r_s} \lambda_s \Xl^{\ml_s} =z^{r_1}a,$$
as desired.
\end{proof}

\begin{corollary} \label{cor2.5} The set map $P\mapsto \pi(P)$ provides a bijection
$$\{ P\in \Hspec \OqMn \mid D_q\notin P\} \longrightarrow \SHspec \OqSLn.$$
\end{corollary}

\begin{proof} Compose the bijection of Proposition \ref{prop2.4} with the bijection
$$\{ P\in \Hspec \OqMn \mid D_q\notin P\} \longrightarrow \Hspec \OqGLn$$
obtained from localization.
\end{proof}

\begin{subsec} \label{subsec2.5} {\bf Generators and normal elements in $\OqSLth$.}
In view of Corollary \ref{cor2.5}, there are exactly $36$ $SH$-prime ideals in $\OqSLth$, also with generating sets encoded in Figure \ref{Hprimegens}. (Here, of course, bullets stand for generators $x_{ij}$.) As in \S\ref{subsec2.3}, 
\begin{enumerate}
\item {\it Each $SH$-prime ideal of $\OqSLth$ has a polynormal regular sequence of generators.}
\end{enumerate}
\end{subsec}

Combining \S\ref{subsec2.3}(4) and \S\ref{subsec2.5}(1) with Theorem \ref{modnormal} yields the following homological information.

\begin{theorem} \label{homolmodHprimes} {\rm (a)} If $P$ is any $H$-prime ideal of $\OqGLth$, then $\OqGLth/P$ is Aus\-lander-Gorenstein and GK-Cohen-Macaulay.

{\rm (b)} If $P$ is any $SH$-prime ideal of $\OqSLth$, then $\OqSLth/P$ is Aus\-land\-er-Gor\-en\-stein and GK-Cohen-Macaulay.
\qed\end{theorem}


\section{Localizations} \label{localizations}

For the remainder of the paper, we take $n=3$, and set $A:= \OqGLth$. Our next step is to identify suitable normal elements with which to build the localizations $A_w$ of \S\ref{redloc}.

\begin{subsec} \label{subsec3.2} {\bf Normal $H$-eigenvectors in factors modulo $H$-primes.} Let $y\in S_3$. Hodges and Levasseur identified certain normal $H$-eigenvectors, labelled $c^{\pm}_{i,y}$, in their factor algebras $\OqSLth/I^{\pm}_y$ \cite[Theorem 2.2.1]{HL}. The corresponding information in our case may be stated as follows:
\begin{enumerate}
\item {\it $[y(1)|1]$ and $[y(12)|12]$ are normal modulo $Q^+_y$.}
\item {\it $[\widetilde{y(1)}|\widetilde{1}]$ and $[\widetilde{y(12)}|\widetilde{12}]$ are normal modulo $Q^-_y$.}
\end{enumerate}
In some cases, the $2\times2$ quantum minors given in (1) or (2) decompose into a product modulo $Q^+_y$ or $Q^-_y$. For example, $[23|12] \equiv X_{21}X_{32}$ modulo $Q^+_{231}$. In such cases, both factors of the product turn out to be normal in the quotient algebra, and we place both in the denominator set we will construct. Similarly, the central element $D_q$ sometimes decomposes modulo $Q^+_y$ or $Q^-_y$, in which case we include its factors in our denominator set. For instance, $D_q \equiv [12|12]X_{33}$ modulo $Q^+_{213}$.

For $w=(w_+,w_-)\in S_3\times S_3$, we use the elements discussed above to generate a multiplicative set $E_w \subset A/Q_w$. Since all the generators will be normal elements, $E_w$ will be a denominator set. The generators for $E_w$ consist of some which are already normal modulo $Q^+_{w_+}$ and some which are normal modulo $Q^-_{w_-}$. It is convenient to use these two types to generate multiplicative sets $E^{\pm}_{w_{\pm}} \subset A/Q^{\pm}_{w_{\pm}}$. The lists of generators are given in Figure \ref{Oregens} below.

It follows from (E1.1) and (E1.3b--e) that  each quantum minor in row $y$ and the second (respectively, third) column of Figure \ref{Oregens} is normal modulo $Q^+_y$ (respectively, $Q^-_y$).

\begin{figure}[h]
$$\begin{matrix}
y  &\qquad&\text{generators for $E^+_y$} &\qquad&\text{generators for $E^-_y$} \\
 &&\text{(modulo $Q^+_y$)} &&\text{(modulo $Q^-_y$)} \\ \\
321  &&X_{31},\, [23|12]  &&[12|23],\, X_{13} \\
231  &&X_{21},\, X_{32}  &&[13|23],\, X_{13} \\
312  &&X_{31},\, [13|12]  &&X_{23},\, X_{12} \\
132  &&X_{11},\, X_{32},\, [23|23]  &&[23|23],\, X_{23},\, X_{11} \\
213  &&X_{21},\, [12|12],\, X_{33}  &&X_{33},\, X_{12},\, [12|12] \\
123  &&X_{11},\, X_{22},\, X_{33}  &&X_{33},\, X_{22},\, X_{11}
\end{matrix}$$
\caption {Generators for denominator sets $E_w^{\pm}$}   \label{Oregens}
\end{figure}
\end{subsec}

\begin{subsec} \label{subsec3.3} {\bf The localizations.}
For $w=(w_+,w_-)\in S_3\times S_3$, define $E_w$ to be the multiplicative subset of $A/Q_w$ generated by the cosets of the elements listed as generators for $E^+_{w_+}$ and $E^-_{w_-}$ in Figure \ref{Oregens}. The generators of $E_w$, and thus all its elements, are normal in $A/Q_w$. Consequently, $E_w$ is a denominator set, and we define
$$A_w := (A/Q_w) [ E_w^{-1} ].$$
The action of $H$ on $A$ induces a rational action on $A/Q_w$, and since $E_w$ consists of $H$-eigen\-vec\-tors, the latter action induces a rational action of $H$ on $A_w$ by $k$-algebra automorphisms.

We next check that none of the generators of $E_w$ is zero. For instance, Figure \ref{Hprimegens} records the fact that $Q_{231,w_-}+ \langle X_{21}\rangle = Q_{132,w_-} \ne Q_{231,w_-}$ for all $w_-$, whence $X_{21}\notin Q_{231,w_-}$. That $X_{31}\notin Q_{312,w_-}$ follows similarly, since $X_{31}\in Q_{312,w_-}$ would imply $X_{21}\in Q_{312,w_-}$ or $X_{32}\in Q_{312,w_-}$, given that $[23|12]\in Q_{312,w_-}$. To see, for example, that $[13|12] \notin Q_{312,w_-}$, observe that $X_{32}\notin Q_{231,w_-^{-1}}$ and apply $S$. That $[23|23] \notin Q_{132,w_-}$ follows from the fact that $D_q \equiv X_{11}[23|23]$ modulo $Q_{132,w_-}$. The other non-membership statements hold for similar reasons.

 Thus, $E_w$ consists of nonzero elements of the domain $A/Q_w$, and therefore the localization map $A/Q_w \rightarrow A_w$ is injective.

In view of Figure \ref{Hprimegens}, we also see that if $J$ is an $H$-prime of $A$ which properly contains $Q_w$, then $J/Q_w$ contains at least one of the generators of $E_w$. It follows that $A_w$ contains no nonzero $H$-primes, and consequently $A_w$ is $H$-simple. This establishes \S\ref{redloc}(1), and \S\ref{redloc}(2) then follows by \cite[Corollaries II.3.9 and II.6.5]{BG}.

We have noted that if $J$ is an $H$-prime of $A$ properly containing $Q_w$, then $(J/Q_w)\cap E_w$ is nonempty. Consequently, $(P/Q_w)\cap E_w$ is nonempty for any $P$ in $\spec_J A$. On the other hand, if $P\in \spec_w A$, the intersection of the $H$-orbit of $P/Q_w$ is zero, so  $P/Q_w$ contains no $H$-eigenvectors of $A/Q_w$. Therefore
$$\spec_w A = \{ P\in \spec A \mid P \supseteq Q_w \text{\ and\ } (P/Q_w) \cap E_w = \varnothing \},$$
and similarly for $\prim_w A$. Consequently, localization provides a bijection $\spec_w A \leftrightarrow \spec A_w$.
Since $\prim_w A$ consists of the maximal elements of $\spec_w A$ \cite[Corollary II.8.5]{BG}, it follows that localization also provides a bijection $\prim_w A \leftrightarrow \max A_w$. The remaining bijections of \S\ref{redloc}(3) follow because $A_w$ is $H$-simple \cite[Corollary II.3.9]{BG}.
\end{subsec}

\begin{subsec} \label{subsec3.4} {\bf Some isomorphisms and anti-isomorphisms.} Some of the localizations $A_w$ are isomorphic or anti-isomorphic to others, via combinations of $\tau$, $S^{\pm1}$, and $\rho$, as follows. First, note that for $y\ne 312$, the automorphism $\tau$ sends the (coset representatives of the) generators for $E^+_y$ listed in Figure \ref{Oregens} to the generators for $E^-_{y^{-1}}$, while for $z\ne 231$, it sends the generators for $E^-_z$ to the generators for $E^+_{z^{-1}}$. Consequently, for $w_+\ne 312$ and $w_-\ne 231$, the isomorphism $A_{Q_{w_+,w_-}} \rightarrow A/Q_{w_-^{-1},w_+^{-1}}$ of \eqref{tauquo} maps $E_{w_+,w_-}$ onto $E_{w_-^{-1},w_+^{-1}}$. Thus, $\tau$ further induces isomorphisms
\begin{align} \label{tauiso}
A_{w_+,w_-} &\ {\overset\tau\longrightarrow}\  A_{w_-^{-1},w_+^{-1}}  &&(w_+\ne 312,\ w_-\ne 231).
\tag{E3.3a}
\end{align}

Next, observe that for $y\ne 231$, the antipode $S$ sends generators for $E^+_y$ to a set of generators for $E^+_{y^{-1}}$, up to units. For instance, in the case $y=321$ we have 
$$S(X_{31})= q^2[23|12]D_q^{-1} \qquad\text{and}\qquad S([23|12])= q^2X_{31}D_Q^{-1},$$ while in the case $y=312$ we have 
\begin{align*}
S(X_{31}) &= q^2[23|12]D_q^{-1} \equiv q^2 X_{21}X_{32}D_q^{-1} \pmod{Q^+_{231}} \\
S([13|12]) &= -qX_{32}D_q^{-1}.
\end{align*}
The cases $y=132,213$ are similar to the latter case. In the case $y=123$, we have 
$$S(X_{ii})= [\itil|\itil]D_q^{-1} \equiv X_{ii}^{-1} \pmod{Q^+_{123}} \qquad\text{for\ } i=1,2,3.$$
Likewise, for $z\ne 312$, $S$ sends generators for $E^-_z$ to generators for $E^-_{z^{-1}}$, up to units. Consequently, for $w_+\ne 231$ and $w_-\ne 312$, the anti-isomorphism $A/Q_{w_+,w_-} \rightarrow A/Q_{w_+^{-1},w_-^{-1}}$ of \eqref{Squo} maps $E_{w_+,w_-}$ onto $E_{w_+^{-1},w_-^{-1}}$, up to units. Thus, $S$ further induces anti-isomorphisms
\begin{align} \label{Santiiso}
A_{w_+,w_-} &\ {\overset{S}\longrightarrow}\  A_{w_+^{-1},w_-^{-1}}  &&(w_+\ne 231,\ w_-\ne 312).
\tag{E3.3b}
\end{align}

Observation of the effect of the anti-automorphism $\rho$ on generators for the denominator sets $E^{\pm}_y$, combined with \eqref{rhoquo}, yields anti-isomorphisms
\begin{align} \label{rhoantiiso}
A_{w_+,w_-} &\ {\overset{\rho}\longrightarrow}\  A_{w_0w_+^{-1}w_0,w_0w_-^{-1}w_0}  &&(w_+\ne 312,\ w_-\ne 231).
\tag{E3.3c}
\end{align}
Finally, we observe that the composition $\rho\tau$ sends generators for $E^{\pm}_y$ to generators for $E^{\mp}_{w_0yw_0}$, with no restriction on $y$. Hence, there are induced anti-isomorphisms
\begin{align} \label{rhotauanti}
A_{w_+,w_-} &\ {\overset{\rho\tau}\longrightarrow}\  A_{w_0w_-w_0,w_0w_+w_0}  &&(\text{all\;} w_+,\ w_-).
\tag{E3.3d}
\end{align}

In the following, we will write $C\gnoc D$ to record that $k$-algebras $C$ and $D$ are anti-iso\-morph\-ic to each other. A list of some (anti-) isomorphisms obtained from \eqref{tauiso}--\eqref{rhotauanti} is given in Figure \ref{equiv}. In each row, the (anti-) automorphism which is used to obtain a given algebra from the first algebra in that row is displayed as a subscript.

\begin{figure}[h]
$$\begin{array}{lllllll}
A_{321,321}  \\
A_{231,321} &&\cong_\tau A_{321,312} &&\gnoc_{S^{-1}} A_{312,321} &&\gnoc_{S^{-1}\tau} A_{321,231}  \\
A_{132,321} &&\cong_\tau A_{321,132} &&\gnoc_\rho A_{213,321} &&\gnoc_{\rho\tau} A_{321,213}  \\
A_{123,321} &&\cong_\tau A_{321,123}  \\
A_{231,231} &&\gnoc_{\rho\tau} A_{312,312}  \\
A_{312,231}  &&\gnoc_S A_{231,312}  \\
 A_{132,231} &&\gnoc_S A_{132,312} &&\gnoc_{\tau S} A_{231,132} &&\cong_{\rho S} A_{213,312}  \\
  &&\cong_{\rho\tau S} A_{231,213} &&\cong_{S^{-1}\tau S} A_{312,132}  &&\gnoc_{S^{-1}\rho S} A_{213,231}  \\
  &&\gnoc_{S^{-1}\rho\tau S} A_{312,213}  \\
A_{123,231} &&\gnoc_S A_{123,312} &&\gnoc_{\tau S} A_{231,123} &&\cong_{S^{-1}\tau S} A_{312,123}  \\
A_{132,132} &&\gnoc_\rho A_{213,213}  \\
A_{213,132} &&\cong_\tau A_{132,213}  \\
A_{123,132} &&\cong_\tau A_{132,123} &&\gnoc_\rho A_{123,213} &&\gnoc_{\rho\tau} A_{213,123}  \\
A_{123,123}  
 \end{array}$$
\caption{Some isomorphisms and anti-isomorphisms among the $A_w$}    \label{equiv}
\end{figure}
\end{subsec}


\section{Centers of localizations} \label{centers}

\begin{subsec} \label{subsec4.1} {\bf Indeterminates.} Let $w=(w_+,w_-)\in S_3\times S_3$, and recall \S\ref{redloc}(2). Our next task is to identify indeterminates for the Laurent polynomial ring $Z(A_w)$. It will be convenient to use the same symbol to denote an element of $A$ as for the corresponding coset in $A/Q_w$. Thus, for $a \in A$ and $e\in E_w$, we write $ae^{-1}$ for the fraction $(a+Q_w)(e+Q_w)^{-1}$ in $A_w$.

We first observe, using (E1.1) and (E1.3b--e), that the elements listed in position $w$ of Figure \ref{indetsZAw} below are central elements of $A_w$.
\end{subsec}

\begin{sidewaysfigure}
$$\begin{matrix}
 &&&321  &231  &312  &132  &213  &123 \\  \\
 &&&D_q  &D_q  &D_q  &&&D_q \\
321  &&&[23|12]X_{13}^{-1}  &[23|12]X_{13}^{-1}  &X_{12}X_{23}X_{31}^{-1}  &D_q  &D_q  &X_{22}[23|12]X_{31}^{-1} \\
 &&&[12|23]X_{31}^{-1}  \\  \\
 &&&D_q  &D_q  &&D_q  &D_q  \\
231  &&&X_{21}X_{32}X_{13}^{-1}  &[13|23]X_{21}^{-1}  &D_q  &X_{11}X_{23}X_{32}^{-1}  &X_{12}X_{33}X_{21}^{-1}  &D_q  \\
 &&&&[12|13]X_{32}^{-1}  \\  \\
  &&&D_q  &&D_q  &D_q  &D_q  \\
312  &&&[12|23]X_{31}^{-1}  &D_q  &[23|13]X_{12}^{-1}  &X_{11}X_{32}X_{23}^{-1}  &X_{21}X_{33}X_{12}^{-1}  &D_q  \\
 &&&&&[13|12]X_{23}^{-1}  \\  \\
 &&&&D_q  &D_q  &X_{11}  &&X_{11}  \\
132  &&&D_q  &X_{11}X_{23}X_{32}^{-1}  &X_{11}X_{32}X_{23}^{-1}  &[23|23]  &D_q  &X_{22}X_{33}  \\
 &&&&&&X_{23}X_{32}^{-1}  \\  \\
 &&&&D_q  &D_q  &&X_{33}  &X_{11}X_{22}  \\
213  &&&D_q  &X_{12}X_{33}X_{21}^{-1}  &X_{21}X_{33}X_{12}^{-1}  &D_q  &[12|12]  &X_{33}  \\
 &&&&&&&X_{12}X_{21}^{-1}  \\  \\
 &&&D_q  &&&X_{11}  &X_{11}X_{22}  &X_{11}  \\
123  &&&X_{22}[12|23]X_{13}^{-1}  &D_q  &D_q  &X_{22}X_{33}  &X_{33}  &X_{22}  \\
 &&&&&&&&X_{33}
 \end{matrix}$$
 \caption{Indeterminates for centers $Z(A_w)$}    \label{indetsZAw}
 \end{sidewaysfigure}

\begin{subsec}  \label{centerqtori} {\bf Centers of quantum tori.}
We reduce the process of determining the centers of the localizations $Z(A_w)$  to calculating centers of quantum tori whose commutation parameters are powers of $q$. Recall that if
\begin{equation}
B= k\langle x_1^{\pm1},\dots,x_m^{\pm1} \mid x_ix_j= q^{a_{ij}}x_jx_i \text{\ for all\ } i,j=1,\dots,m\rangle,  \tag{E4.2a}
\end{equation}
where $(a_{ij})$ is an antisymmetric $m\times m$ integer matrix, then $Z(B)$ is spanned by the monomials in the $x_i$ it contains, and $Z(B)$ is a Laurent polynomial ring $k[z_1^{\pm1},\dots,z_d^{\pm1}]$ for some monomials $z_j$ (e.g., \cite[Lemma 1.2]{GLet}, \cite[Lemma 2.4(a)]{Van}). The central monomials are determined as follows:
\begin{equation}
x_1^{s_1} x_2^{s_2} \cdots x_m^{s_m} \in Z(B) \quad\iff\quad  \sum_{j=1}^m a_{ij}s_j = 0 \quad\text{for\ } i=1,\dots,m,  \tag{E4.2b}
\end{equation}
because of our assumption that $q$ is not a root of unity.
\end{subsec}

\begin{lemma} \label{lem4.2} Let $w=(w_+,w_-)\in S_3\times S_3$, and let $z_1,\dots,z_d$ be the elements of $A_w$ listed in position $w$ of Figure {\rm\ref{indetsZAw}}. Then $Z(A_w)$ is a Laurent polynomial ring of the form
$$Z(A_w)= k[z_1^{\pm1},\dots,z_d^{\pm1}].$$
\end{lemma}

\begin{proof} (a) There are 9 cases in which $A_w$ is a quantum torus, namely, when
\begin{align*}
w=\; &(231,123),\; (132,213),\; (132,123),\; (213,132),\\
 &(213,123),\; (123,312),\; (123,132),\; (123,213),\; (123,123).
 \end{align*}
In all of these cases, $D_q= X_{11}X_{22}X_{33}$ in $A/Q_w$, whence $X_{11}$, $X_{22}$, $X_{33}$ are invertible in $A/Q_w$. Any other $X_{ij}\notin Q_w$ becomes invertible in $A_w$ because it occurs in $E_w$. Thus, $A_w$ takes the form of a quantum torus on the generators $X_{ij}^{\pm1}$ for those $X_{ij}\notin Q_w$.

In the case $w=(123,123)$, the algebra $A_w$ is commutative, equal to a Laurent polynomial ring $k[X_{11}^{\pm1},X_{22}^{\pm1},X_{33}^{\pm1}]$.

In the case $w=(231,123)$, the algebra $A_w$ is generated by $X_{11}^{\pm1}$, $X_{21}^{\pm1}$, $X_{22}^{\pm1}$, $X_{32}^{\pm1}$, $X_{33}^{\pm1}$. A monomial in these generators is central if and only if it is of the form $X_{11}^s X_{22}^s X_{33}^s$ for some $s\in\ZZ$. Thus, $Z(A_w)= k[(X_{11}X_{22}X_{33})^{\pm1}]$, which we rewrite in the form $Z(A_w)= k[D_q^{\pm1}]$. 

The remaining 7 cases listed above are analyzed in the same manner.

(b) There are 8 cases in which $A_w$ can be presented as a localization of a skew-Laurent extension of $\Oq(GL_2(k))$:
\begin{align*}
w=\; &(321,123),\; (231,132),\; (231,213),\; (132,312),  \\
 &(132,132),\; (213,312),\; (213,213),\; (123,321).
 \end{align*}
We deal with the case $w= (321,123)$ as follows. While $A_w$ is not itself a quantum torus, it can be localized to one. Namely, the powers of $X_{21}$ form a denominator set, and the localization $A'_w := A_w[X_{21}^{-1}]$ is a quantum torus on the generators $X_{11}^{\pm1}$, $X_{21}^{\pm1}$, $X_{22}^{\pm1}$, $X_{31}^{\pm1}$, $[23|12]^{\pm1}$, $X_{33}^{\pm1}$. Some of the work of checking this can be avoided by first forming a quantum torus $B$ as in (E4.2a) with $m=6$ and
$$(a_{ij}) = \left( \begin{smallmatrix} 0 &1 &0 &1 &1 &0\\ -1 &0 &1 &1 &0 &0\\ 0 &-1 &0 &0 &0 &0\\ -1 &-1 &0 &0 &0 &1\\ -1 &0 &0 &0 &0 &1\\ 0 &0 &0 &-1 &-1 &0 \end{smallmatrix} \right);$$
observing that there is a $k$-algebra homomorphism $\phi: B\rightarrow \Fract(A/Q_w)$ sending $x_1,\dots,x_6$ to $X_{11}$, $X_{21}$, $X_{22}$, $X_{31}$, $[23|12]$, $X_{33}$, respectively; and checking that $\phi(B)$ is generated by $A_w\cup \{X_{21}^{-1}\}$. (That $X_{32}\in \phi(B)$ follows from the identity $X_{32}= X_{21}^{-1}(qX_{22}X_{31}+[23|12])$.) Now $B$ is a domain with GK-dimension 6, while $\GKdim(\phi(B)) \ge \GKdim(A/Q_w) =6$, from which we conclude that $\ker\phi=0$. Thus, $\phi$ maps $B$ isomorphically onto $A'_w$.
From (E4.2b), we find that a monomial $X_{11}^a X_{21}^b X_{22}^c X_{31}^d [23|12]^e X_{33}^f$ is central if and only if $a=f=c+d$, $b=0$, and $e=-d$. Consequently, $Z(A'_w)$ can be written as a Laurent polynomial ring in indeterminates $X_{11}X_{22}X_{33}$ and $X_{22}X_{31}^{-1}[23|12]$. For convenience, we rewrite these as $D_q$ and $X_{22}[23|12]X_{31}^{-1}$. Since these are elements of $A_w$, we see that $Z(A'_w) \subseteq A_w$, and therefore $Z(A'_w)= Z(A_w)$. This verifies the entry in position $w$ of Figure \ref{indetsZAw}. 

The other 7 cases can be analyzed in the same manner. In fact, it suffices to deal with the cases $(132,312)$ and $(132,132)$, in view of the (anti-) isomorphisms
\begin{align*}
A_{321,123} &\cong_\tau A_{123,321}  \\
A_{231,132} &\cong_\tau A_{132,312} \gnoc_\rho A_{213,312} \cong_\tau A_{231,213}  \\
A_{132,132} &\gnoc_\rho A_{213,213}
\end{align*}
given by \eqref{tauiso}, \eqref{rhoantiiso}. For the mentioned cases, the following localizations $A'_w$ of $A_w$ can be used: 
\begin{align*}
A_{132,312}[X_{33}^{-1}] &&&A_{132,132}[X_{33}^{-1}] 
\end{align*}

(c) Six of the cases analyzed above yield additional cases via the following anti-iso\-morph\-isms obtained from \eqref{Santiiso}:
\begin{align*}
A_{312,213} &\gnoc_S A_{231,213}  &A_{312,123} &\gnoc_S A_{231,123}  &A_{123,231} &\gnoc_S A_{123,312}  \\
A_{312,132} &\gnoc_S A_{231,132}  &A_{132,231} &\gnoc_S A_{132,312}  &A_{213,231} &\gnoc_S A_{213,312}
\end{align*}
For instance, taking $w=(312,213)$, the anti-isomorphism $A_w \rightarrow A_{231,213}$ of \eqref{Santiiso} sends 
\begin{align*}
D_q^{-1} &\longmapsto D_q  &q^2[23|13]^{-1}D_q[12|12]D_q^{-1}[13|23]D_q^{-1} &\longmapsto X_{12}X_{33}X_{21}^{-1}.
\end{align*}
In $A_w$, we have $[23|13]^{-1}D_q[12|12]D_q^{-1}[13|23]D_q^{-1}= X_{21}^{-1}X_{33}^{-1}X_{12}$, and so $Z(A_w)$ can be written as a Laurent polynomial ring in indeterminates $D_q$ and $X_{21}X_{33}X_{12}^{-1}$. This establishes the case $w=(312,213)$, and the other 5 can be analyzed in the same manner.

(d) Next, we consider 4 cases in which $A_w$ is isomorphic to a localization of a skew-Laurent extension of $\Oq(M_{2,3}(k))$:
$$w= (321,132),\; (321,213),\; (132,321),\; (213,321).$$
If $w= (321,132)$, the localization $A'_w := A_w[X_{32}^{-1},X_{33}^{-1}]$ is a quantum torus on the generators $X_{11}^{\pm1}$, $X_{23}^{\pm1}$, $X_{31}^{\pm1}$, $X_{32}^{\pm1}$, $X_{33}^{\pm1}$, $[23|12]^{\pm1}$, $[23|23]^{\pm1}$. Using (E4.2b), we compute that the central monomials in $A'_w$ are the powers of $X_{11}[23|23]= D_q$, whence $Z(A'_w)= k[D_q^{\pm1}]$, and thus $Z(A_w)= k[D_q^{\pm1}]$.The remaining cases follow from this one via the (anti-) isomorphisms
\begin{align*}
A_{321,132} &\cong_\tau A_{132,321}  \gnoc_\rho A_{213,321} \cong_\tau A_{321,213}
\end{align*}
given by \eqref{tauiso}, \eqref{rhoantiiso}.

(e) Four cases in which $Q_w$ has height 2 remain:
$$w= (231,231),\; (231,312),\; (312,231),\; (312,312).$$
If $w=(231,231)$, the localization $A'_w:= A_w[X_{33}^{-1},[23|23]^{-1}]$ is a quantum torus on the generators $X_{13}^{\pm1}$, $X_{21}^{\pm1}$, $X_{32}^{\pm1}$, $X_{33}^{\pm1}$, $[13|23]^{\pm1}$, $[23|23]^{\pm1}$, $D_q^{\pm1}$. To see that $X_{11}$, $X_{12}$, $X_{22}$, $X_{23}$ lie in the given quantum torus, observe, using (E1.3a), that
\begin{align*}
X_{11} &= [23|23]^{-1} \bigl( D_q+ q^{-1}[13|23]X_{21} \bigr)  &X_{12} &= \bigl( [13|23] +qX_{13}X_{32} \bigr) X_{33}^{-1}  \\
X_{22} &= q[13|23]^{-1}[23|23]X_{12}  &X_{23} &= q[13|23]^{-1}[23|23]X_{13}
\end{align*}
in $\Fract(A/Q_w)$. With the help of (E4.2b), we compute that $Z(A'_w)$ is a Laurent polynomial ring in the indeterminates $D_q$, $[13|23]X_{21}^{-1}$, and $X_{13}X_{32}^{-1}[13|23]^{-1}D_q$. The last of these is chosen to take advantage of the identity
\begin{equation}
X_{13}D_q = [12|13][13|23] -q [13|13][12|23],  \tag{E4.3a}
\end{equation}
which is obtained by applying $S$ to the identity $[12|23]= X_{12}X_{23} -q X_{13}X_{22}$ and multiplying by $q^2D_q^2$. In particular,
\begin{equation}
X_{13}D_q \equiv  [12|13][13|23] \quad \text{modulo\ } \bigl\langle [12|23] \bigr\rangle,  \tag{E4.3b}
\end{equation}
whence $X_{13}X_{32}^{-1}[13|23]^{-1}D_q= [12|13]X_{32}^{-1}$ in $A_w$. Consequently, $Z(A'_w)$ is a Laurent polynomial ring in the indeterminates listed in position $w$ of Figure \ref{indetsZAw}. Then $Z'(A_w)= Z(A_w)$, completing this case.

The case $(231,312)$ is analyzed in the same manner, and the other two cases follow via the anti-isomorphisms $A_{312,231} \gnoc_S A_{231,312}$ and $A_{231,231} \gnoc_{\rho\tau} A_{312,312}$ of \eqref{Santiiso}, \eqref{rhotauanti}.

(f) There are 5 cases remaining:
$$w= (321,321),\; (321,231),\; (321,312),\; (231,321),\; (312,321).$$
If $w=(321,312)$, the localization $A'_w := A_w[X_{11}^{-1},X_{21}^{-1},[12|12]^{-1}]$ is a quantum torus on the generators $X_{11}^{\pm1}$, $X_{12}^{\pm1}$, $X_{21}^{\pm1}$, $X_{23}^{\pm1}$, $X_{31}^{\pm1}$, $[12|12]^{\pm1}$, $[23|12]^{\pm1}$, $D_q^{\pm1}$, which we analyze as above.

The cases $(321,231)$, $(231,321)$, $(312,321)$ follow via the (anti-) iso\-morph\-isms
$$A_{312,321} \gnoc_S A_{231,321} \cong_\tau A_{321,312} \gnoc_{S^{-1}} A_{321,231}$$
of \eqref{tauiso}, \eqref{Santiiso}.

In the final case, $w=(321,321)$, the localization 
$$A'_w := A_w[X_{11}^{-1},X_{12}^{-1},X_{21}^{-1},[12|12]^{-1}]$$
is a quantum torus on the generators $X_{11}^{\pm1}$, $X_{12}^{\pm1}$, $X_{13}^{\pm1}$, $X_{21}^{\pm1}$, $X_{31}^{\pm1}$, $[12|12]^{\pm1}$, $[12|23]^{\pm1}$, $[23|12]^{\pm1}$, $D_q^{\pm1}$, which we analyze as above. 
\end{proof}


\section{Primitive ideals} \label{primitives}

\begin{subsec} \label{subsec5.1} {\bf Generators.} Let $w\in S_3\times S_3$, let $z_1,\dots,z_d$ be the elements of $A_w$ listed in position $w$ of Figure \ref{indetsZAw}, and let $\alpha_1,\dots,\alpha_d$ in $\kx$. By \S\ref{redloc}(3),
\begin{enumerate}
\item {\it For any $l\le d$, the elements $z_1-\alpha_1 ,\dots, z_l-\alpha_l$ generate a prime ideal of $A_w$. It is a maximal ideal if and only if $l=d$.}
\item {\it For any $l\le d$, the ideal
\begin{equation} \label{Pprime}
P'= P'_w(\alpha_1,\dots,\alpha_l) := \bigl( (z_1-\alpha_1)A_w +\cdots+ (z_l-\alpha_l)A_w \bigr) \cap (A/Q_w)  \tag{E5.1}
\end{equation}
is a prime ideal of $A/Q_w$. It is primitive {\rm(}whence $P'=P/Q_w$ for some $P\in \prim_w A${\rm)} if and only if $l=d$.}
\item {\it If $k$ is algebraically closed, then every quotient $P/Q_w$, for $P\in \prim_w A$, is an ideal of the form \eqref{Pprime} for $l=d$ and some $\alpha_i\in \kx$.}
\end{enumerate}

We wish to identify generators for the prime ideals \eqref{Pprime}. Each $z_i$ can be expressed as a fraction $e_if_i^{-1}$ for some $e_i,f_i\in E_w$, and so $P'= P''A_w\cap (A/Q_w)$ where
$$P'' = (e_1-\alpha_1f_1)(A/Q_w) +\cdots+ (e_l-\alpha_lf_l)(A/Q_w).$$
Thus, $P'$ contains $P''$ and these two ideals are equal up to $E_w$-torsion. They will be equal -- thus providing a set of generators for $P'$ -- if and only if $P'/P''$ is $E_w$-torsionfree. To establish the latter condition, it will suffice to show that generators for $E_w$ are regular modulo $P''$.

Elements which we shall use to generate the above ideals are given in Figure \ref{polyprimegens}, where $\alpha$, $\beta$, $\gamma$ denote arbitrary nonzero elements of $k$. Each entry in position $w$ of this table is of the form $e-\lambda f$ for some $e,f\in E_w$ and $\lambda\in\kx$, where $ef^{-1}$ appears in position $w$ of Figure \ref{indetsZAw}. Thus, $ef^{-1}$ is a central element of $A_w$ and $f$ is a normal element of $A/Q_w$, from which it follows that $e-\lambda f$ is normal in $A/Q_w$. To summarize:
\begin{enumerate}
\item[(4)] {\it The elements $c_1,\dots,c_d$ listed in position $w$ of Figure {\rm\ref{polyprimegens}} are normal elements of $A/Q_w$.}
\item[(5)] {\it For $l\le d$, the elements $c_1,\dots,c_l$ generate a prime ideal of $A_w$.}
\end{enumerate}
\end{subsec}

\begin{sidewaysfigure} 
$$\begin{matrix}
 &&&321  &231  &312  &132  &213  &123 \\  \\
 &&&D_q-\alpha  &D_q-\alpha  &D_q-\alpha  &&&D_q-\alpha \\
321  &&&[23|12]-\beta X_{13}  &[23|12]-\beta X_{13}  &X_{12}X_{23}-\beta X_{31}  &D_q-\alpha  &D_q-\alpha  &X_{22}[23|12]-\beta X_{31} \\
 &&&[12|23]-\gamma X_{31}  \\  \\
 &&&D_q-\alpha  &D_q-\alpha  &&D_q-\alpha  &D_q-\alpha  \\
231  &&&X_{21}X_{32}-\beta X_{13}  &[13|23]-\beta X_{21}  &D_q-\alpha  &X_{11}X_{23}-\beta X_{32}  &X_{12}X_{33}-\beta X_{21}  &D_q-\alpha  \\
 &&&&[12|13]-\gamma X_{32}  \\  \\
  &&&D_q-\alpha  &&D_q-\alpha  &D_q-\alpha  &D_q-\alpha  \\
312  &&&[12|23]-\beta X_{31}  &D_q-\alpha  &[23|13]-\beta X_{12}  &X_{11}X_{32}-\beta X_{23}  &X_{21}X_{33}-\beta X_{12}  &D_q-\alpha  \\
 &&&&&[13|12]-\gamma X_{23}  \\  \\
 &&&&D_q-\alpha  &D_q-\alpha  &X_{11}-\alpha  &&X_{11}-\alpha  \\
132  &&&D_q-\alpha  &X_{11}X_{23}-\beta X_{32}  &X_{11}X_{32}-\beta X_{23}  &[23|23]-\beta  &D_q-\alpha  &X_{22}X_{33}-\beta   \\
 &&&&&&X_{23}-\gamma X_{32}  \\  \\
 &&&&D_q-\alpha  &D_q-\alpha  &&X_{33}-\alpha  &X_{11}X_{22}-\alpha  \\
213  &&&D_q-\alpha  &X_{12}X_{33}-\beta X_{21}  &X_{21}X_{33}-\beta X_{12}  &D_q-\alpha  &[12|12]-\beta   &X_{33}-\beta   \\
 &&&&&&&X_{12}-\gamma X_{21}  \\  \\
 &&&D_q-\alpha  &&&X_{11}-\alpha  &X_{11}X_{22}-\alpha  &X_{11}-\alpha  \\
123  &&&X_{22}[12|23]-\beta X_{13}  &D_q-\alpha  &D_q-\alpha  &X_{22}X_{33}-\beta   &X_{33}-\beta   &X_{22}-\beta   \\
 &&&&&&&&X_{33}-\gamma
 \end{matrix}$$
\caption {Generators for some prime ideals in factor algebras $A/Q_w$}   \label{polyprimegens}
\end{sidewaysfigure}

\begin{subsec} \label{subsec5.2} To aid in showing that the elements displayed in Figure \ref{polyprimegens} generate prime ideals in various factor algebras of $A$, we record some definitions and useful observations.

Suppose that $R$ is a noetherian ring, $\sigma$ an automorphism of $R$, and $a\in Z(R)$. The {\it generalized Weyl algebra\/} constructed from these data is  the ring $R(\sigma,a)$ generated by $R$ together with two elements $x$, $y$ subject to the relations
\begin{align*}
yx &= a  &xy &= \sigma(a)  &xr &= \sigma(r)x  &yr &= \sigma^{-1}(r)y
\end{align*}
for $r\in R$. If $R$ is a domain and $a\ne 0$, then  $R(\sigma,a)$ is a domain \cite[Proposition 1.3(2)]{Bav}.

Let $\pi: A\rightarrow \OqSLth$ be the quotient map, and $w\in S_3\times S_3$. By Proposition \ref{prop2.4}, $\pi(Q_w)$ is a prime ideal of $\OqSLth$, and so its inverse image, $Q_w+\langle D_q-1\rangle$, is a prime ideal of $A$. Given any $\alpha\in\kx$, we can choose $h\in H$ such that $h(D_q)= \alpha^{-1}D_q$, for instance $h= (\alpha^{-1},1,1,1,1,1)$. Then $h(Q_w+\langle D_q-1\rangle)= Q_w+\langle D_q-\alpha\rangle$, and we conclude that
\begin{enumerate}
\item {\it $Q_w+\langle D_q-\alpha\rangle$ is a prime ideal of $A$ for any $w\in S_3\times S_3$ and $\alpha\in\kx$.}
\end{enumerate}

Proposition \ref{prop2.4} also shows that $\pi(Q_w)\ne \pi(Q_v)$ for any distinct $w,v\in S_3\times S_3$. In particular, $\pi$ must preserve strict inclusions among the $Q_w$, whence
\begin{enumerate}
\item[(2)] {\it If $w,v\in S_3\times S_3$ and $Q_w \subsetneq Q_v$, then $Q_w+ \langle D_q-1\rangle \subsetneq Q_v+ \langle D_q-1 \rangle$. }
\end{enumerate}
This statement is useful in showing that certain elements do not belong to prime ideals of the given form. For instance, if $x\in A$ and $Q_w+ \langle x\rangle= Q_v$ for some $v\ne w$, statement (2) implies that $x\notin Q_w+ \langle D_q-1\rangle$. 

As a particular example, 
$$X_{12}\,,\, X_{23} \notin Q_{231,312}+ \langle D_q-1\rangle= Q_{321,312}+ \langle D_q-1,\, X_{31} \rangle.$$
Since the displayed ideal is prime (by (1)), it does not contain $X_{12}X_{23}$,
and thus
$$X_{12}X_{23}- X_{31} \notin Q_{321,312}+ \langle D_q-1,\, X_{31} \rangle.$$
Similarly, $X_{12}X_{23}-X_{31} \notin Q_{123,312}+ \langle D_q-1\rangle$. Since $Q_{321,312} \subset Q_{123,312}$ and $[23|12] \in Q_{123,312}$, it follows that
$$X_{12}X_{23}- X_{31} \notin Q_{321,312}+ \langle D_q-1,\, [23|12] \rangle.$$

The following easy observation will help us check that certain elements are regular modulo certain ideals.
\begin{enumerate}
\item[(3)] {\it Let $x$ and $y$ be nonzero normal elements in a domain $B$. If $y$ is regular modulo $\langle x\rangle$ {\rm(}e.g., if $\langle x\rangle$ is completely prime and $y\notin \langle x\rangle${\rm)}, then $x$ is regular modulo $\langle y\rangle$.}
\end{enumerate}
For, if $b\in B$ and $xb\in \langle y\rangle$, then $xb=cy$ for some $c\in B$. Regularity of $y$ modulo $\langle x\rangle$ implies that $c=xd$ for some $d\in B$, and thus $xb=xdy$. Cancelling $x$ (valid because $B$ is a domain) yields $b=dy\in \langle y\rangle$. Similarly, $bx\in \langle y\rangle$ implies $b\in \langle y\rangle$.

In applying (3) in factor algebras $B$ of $A$, we continually rely on the fact that all prime ideals are completely prime (\S\ref{subsec1.1}(1)).

For example, we know from (1) that $B:= A/\bigl( Q_{321,312}+ \langle D_q-1 \rangle\bigr)$ is a domain and that $X_{31}B$ is a prime ideal of $B$. As shown above,  $X_{12}X_{23}-X_{31}$ is not in $X_{31}B$. Thus, (3) implies that  $X_{31}$ is regular modulo $(X_{12}X_{23}-X_{31})B$. We also saw that  $X_{12}X_{23}-X_{31}$ is not in $[23|12]B$. The latter is a prime ideal of $B$, because
$$Q_{321,312}+ \langle D_q-1,\, [23|12] \rangle= Q_{312,312}+ \langle D_q-1 \rangle.$$
Consequently, (3) implies that  $[23|12]$ is regular modulo $(X_{12}X_{23}-X_{31})B$. Restating our information in $A$, we obtain that $X_{31}$ and $[23|12]$ are regular modulo $Q_{321,312}+ \langle D_q-1,\, X_{12}X_{23}-X_{31}\rangle$.

Finally, we record one extension of (3):
\begin{enumerate}
\item[(4)] {\it Let $x$, $y$, $z$ be nonzero normal elements in a domain $B$. If $y$ is regular modulo $\langle x\rangle$ and $z$ is regular modulo $\langle x,y\rangle$, then $x$ is regular modulo $\langle y,z\rangle$.}
\end{enumerate}
For, if $b\in B$ and $xb\in \langle y,z\rangle$, then $xb= c_1y+c_2z$ for some $c_i\in B$, and $c_2z\in \langle x,y\rangle$. By hypothesis, $c_2= xd_1+yd_2$ for some $d_j\in B$, and $x(b-d_1z)= c_1y+ yd_2z\in \langle y\rangle$. Since $x$ is regular modulo $\langle y\rangle$ by (3), it follows that $b-d_1z\in \langle y\rangle$ and thus $b\in \langle y,z\rangle$. Similarly, $bx\in \langle y,z\rangle$ implies $b\in \langle y,z\rangle$.

For example, from (1) the algebra $B:= A/ \bigl( Q_{231,231}+ \langle D_q-1\rangle \bigr)$ is a domain and $X_{21}B$ is a prime ideal. By (2),  $X_{32}$ is not in $Q_{132,213}+ \langle D_q-1\rangle$. Since the latter ideal contains $Q_{231,231}$ as well as $X_{21}$ and $[12|13]$, we see that $[12|13] -X_{32} \notin Q_{231,231}+ \langle D_q-1,\, X_{21}\rangle$. Thus,  $[12|13]-X_{32}$ is not in $X_{21}B$, so it is regular modulo $X_{21}B$. Similarly, by inspecting $Q_{123,213}$ we see that  $[13|23]-X_{21}$ is not in $X_{21}B + ([12|13]-X_{32})B$. This ideal of $B$ is prime, as we shall prove in the case $(132,231)$ of Lemma \ref{lem5.3}. Once that is established, (4) will imply that $X_{21}$ is regular modulo
$$Q_{231,231}+ \langle D_q-1,\, [12|13]-X_{32},\, [13|23]-X_{21} \rangle.$$
\end{subsec}

\begin{lemma} \label{lem5.3} Let $w\in S_3\times S_3$, and let $a_1,\dots,a_d$ be the elements of $A/Q_w$ listed in position $w$ of Figure {\rm\ref{polyprimegens}} {\rm(}for some choices of $\alpha,\beta,\gamma \in \kx${\rm)}. For $l=1,\dots,d$, the elements $a_1,\dots,a_l$ generate a prime ideal of $A/Q_w$.
\end{lemma}

\begin{proof} \lowunder{$l=1$}. For any $\alpha\in\kx$, \S5.2(1) implies that $\langle D_q-\alpha\rangle$ is a prime ideal of $A/Q_w$. This establishes the case $l=1$ of the lemma for 29 of the 36 choices of $w$. We shall deal with the other 7 choices in cases (a) and (b) below. For now, note also that the lemma is complete in the following 12 cases:
\begin{align*}
w = &\;(321,132),\; (321,213),\; (231,312),\; (231,123),\; (312,231),\; (312,123), \\
 &\;(132,321),\; (132,213),\; (213,321),\; (213,132),\; (123,231),\; (123,312).
\end{align*}

(a) As in the proof of Lemma \ref{lem4.2}, we next address the 9 cases in which $A_w$ is a quantum torus. Four of these cases are covered under $(l=1)$ above, leaving the following 5 cases:
$$w= (132,123),\; (213,123),\; (123,132),\; (123,213),\; (123,123).$$

If $w= (123,123)$, the algebra $A/Q_w$ is a commutative Laurent polynomial ring with indeterminates $X_{11}$, $X_{22}$, $X_{33}$, and the desired results are clear. If $w= (132,123)$, then
\begin{gather*}
A/\bigl( Q_w+ \langle X_{11}-\alpha \rangle \bigr) \cong k\langle x^{\pm1},y,z^{\pm1} \mid xy=qyx,\ xz=zx,\ yz=qzy \rangle \\
A/\bigl( Q_w+ \langle X_{11}-\alpha,\; X_{22}X_{33}-\beta \rangle \bigr) \cong k\langle x^{\pm1},y \mid xy=qyx \rangle,
\end{gather*}
both of which are localizations of quantum affine spaces and so are domains. Thus, the ideals $\langle X_{11}-\alpha \rangle$ and $\langle X_{11}-\alpha,\; X_{22}X_{33}-\beta \rangle$ are prime ideals of $A/Q_w$. The other 3 cases follow the same pattern.

(b) Next, consider the 8 cases of Lemma \ref{lem4.2}(b):
\begin{align*}
w=\; &(321,123),\; (231,132),\; (231,213),\; (132,312),  \\
 &(132,132),\; (213,312),\; (213,213),\; (123,321).
 \end{align*}

\lowunder{$w=(132,132)$}. The relevant quotient algebras have the following forms:
\begin{multline}
{\begin{aligned}
A/\bigl( Q_w+ \langle X_{11}-\alpha\rangle \bigr) &\cong \Oq(GL_2(k))  \\
A/\bigl( Q_w+ \langle X_{11}-\alpha,\; [23|23]-\beta \rangle \bigr) &\cong \Oq(SL_2(k))  \\
A/\bigl( Q_w + \langle X_{11}-\alpha,\; [23|23]-\beta,\; X_{23}-\gamma X_{32} \rangle \bigr) &\cong
\end{aligned}  \notag}  \\
 k\langle x,z,y\mid xz=qzx,\; zy=qyz,\; xy-q\gamma z^2= yx-q^{-1}\gamma z^2= \beta \rangle.
\end{multline}
(In the third case, $x$, $z$, $y$ correspond to $X_{22}$, $X_{32}$, $X_{33}$, respectively.)
The first two algebras are known domains. The third is a generalized Weyl algebra $k[z](\sigma,a)$ where $\sigma$ is the automorphism of $k[z]$ sending $z\mapsto qz$ and $a=q^{-1}\gamma z^2 +\beta$, and thus it too is a domain. This establishes the case $w=(132,132)$. 

The case $w=(213,213)$ is handled via the induced anti-isomorphism 
$$A/Q_{132,132} \gnoc_\rho A/Q_{213,213}.$$
These results, together with (a), cover the remaining cases of $(l=1)$.

\lowunder{$w= (231,132)$}. The quotient $A/\bigl( Q_w+ \langle D_q-\alpha,\; X_{11}X_{23}- \beta X_{32} \rangle \bigr)$ is isomorphic to the $k$-algebra with generators $w$, $x$, $z$, $y$, $D^{\pm1}$ and relations
\begin{gather*}
\begin{align*}
wx &= qxw  &\qquad\qquad wz &= zw  &\qquad\qquad wy &= yw  \\
wD &= qDw  &xz &= qzx  &zy &= qyz  \\
xD &= Dx  &zD &= Dz  &yD &= Dy
\end{align*}  \\
xy- q\alpha^{-1}\beta z^2D = yx-q^{-1}\alpha^{-1}\beta z^2D = D.
\end{gather*}
(Here $w$, $x$, $z$, $y$, $D$ correspond to $X_{21}$, $X_{22}$, $X_{32}$, $X_{33}$, $[23|23]$, respectively.)
This algebra is a domain because it can be expressed as a skew polynomial extension of a generalized Weyl algebra in the form $\bigl( k[z,D^{\pm1}](\sigma,a) \bigr) [w;\rho]$ where the automorphism $\sigma$ sends $z\mapsto qz$ and $D\mapsto D$, the element $a= q^{-1}\alpha^{-1}\beta z^2D +D$, and the automorphism $\rho$ sends
\begin{align*}
z &\mapsto z  &D &\mapsto qD  &x &\mapsto qx  &y &\mapsto y.
\end{align*}

In view of the induced (anti-) isomorphisms
\begin{multline*}
A/\bigl( Q_{231,132}+ \langle D_q-\alpha,\; X_{11}X_{23}- \beta X_{32} \rangle \bigr)
 \cong_\tau  \\
  \shoveright{A/\bigl( Q_{132, 312}+ \langle D_q-\alpha,\; X_{11}X_{32}- \beta X_{23} \rangle \bigr)}  \\
\shoveleft{A/\bigl( Q_{231,132}+ \langle D_q-\alpha,\; X_{11}X_{23}- \beta X_{32} \rangle \bigr)
 \gnoc_\rho}  \\
 \shoveright{A/\bigl( Q_{231,213}+ \langle D_q-\alpha,\; X_{12}X_{33}-\beta X_{21} \rangle \bigr)}  \\
 \shoveleft{A/\bigl( Q_{231,213}+ \langle D_q-\alpha,\; X_{12}X_{33}- \beta X_{21} \rangle \bigr)
 \cong_\tau}  \\
  A/\bigl( Q_{213, 312}+ \langle D_q-\alpha,\; X_{21}X_{33}- \beta X_{12} \rangle \bigr),
 \end{multline*}
the cases $(132,312)$, $(231,213)$, and $(213,312)$ follow.

\lowunder{$w= (321,123)$}. Here we find that
$$A/\bigl( Q_w+ \langle D_q-\alpha,\; X_{22}[23|12]- \beta X_{31} \rangle \bigr) \cong \bigl( k[z^{\pm1},D](\sigma,a) \bigr) [w;\rho],$$
where $a= q^{-1}\beta^{-1} z^2D +D$ and $\sigma$, $\rho$ act as in the case $(231,132)$. Finally, the case $(123,321)$ follows from this one by applying $\tau$.

(c) Of the 6 cases of Lemma \ref{lem4.2}(c), $(312,123)$ and $(123,231)$ are covered under $(l=1)$, leaving
$$w= (312,132),\; (312,213),\; (132,231),\; (213,231).$$

The anti-isomorphism $A/Q_{231,132} \rightarrow A/Q_{312,132}$ of \eqref{Squo} sends $D_q-\alpha$ to $D_q^{-1}-\alpha$ and
\begin{multline*}
X_{11}X_{23}-\beta X_{32}= [12|13]-\beta X_{32} \longmapsto  \\
-q^{-1}X_{23}D_q^{-1}+ \beta q [13|12]D_q^{-1}= -q^{-1}X_{23}D_q^{-1}+ \beta q X_{11}X_{32} D_q^{-1},
\end{multline*}
and consequently
\begin{multline*}
A/\bigl( Q_{231,132}+ \langle D_q-\alpha,\; X_{11}X_{23}- \beta X_{32} \rangle \bigr)   \gnoc \\
 A/\bigl( Q_{312,132}+ \langle D_q-\alpha^{-1},\; X_{11}X_{32}- \beta^{-1}q^{-2} X_{23} \rangle \bigr).
\end{multline*}
Since $\alpha^{-1}$ and $\beta^{-1}q^{-2}$ run through all choices of nonzero scalars in $k$, the case $(312,132)$ thus follows from the case $(231,132)$. Similarly, the case $(312,213)$ follows from the case $(231,213)$.

The cases $(132,231)$ and $(213,231)$ now follow from the cases $(312,132)$ and $(312,213)$ via $\tau$.

(d) The four cases of Lemma \ref{lem4.2}(d), namely
$$w= (321,132),\; (321,213),\; (132,321),\; (213,321),$$
are all covered under $(l=1)$.

(e)(f) Of the 9 cases considered in Lemma \ref{lem4.2}(e)(f), $(231,312)$ and $(312,231)$ are covered under $(l=1)$, leaving
\begin{align*}
w=\; &(321,321),\; (321,231),\; (321,312),\\
 &(231,321),\; (231,231),\; (312,321),\; (312,312).
 \end{align*}

\lowunder{$w= (321,312)$}. We must show that the ideal
$$P := Q_w + \langle D_q-\alpha,\; X_{12}X_{23}- \beta X_{31} \rangle$$
of $A$ is prime. If $h= (\delta,1,\varepsilon,1,1,1) \in H$, where $\delta,\varepsilon \in \kx$ are chosen so that $\delta\varepsilon= \alpha^{-1}$ and $\delta^{-1}\varepsilon= \beta$, then
$$h \bigl( Q_w + \langle D_q-1,\; X_{12}X_{23}- X_{31} \rangle \bigr) = Q_w + \langle D_q-\alpha,\; X_{12}X_{23}- \beta X_{31} \rangle.$$
Hence, we may assume that $\alpha=\beta=1$. In view of the discussion in \S5.1, it will suffice to show that the generators of $E_w$ are all regular modulo $P$. In the present case, these generators are
$$X_{31}\,,\; [23|12]\,,\; X_{23}\,,\; X_{12}\,.$$
To show that these four elements are regular modulo $P$, we verify the corresponding properties for the ideal $P':= P/\bigl( Q_w + \langle D_q-1\rangle \bigr)$ in the domain 
$$A':= A/\bigl( Q_w + \langle D_q-1\rangle \bigr).$$

That $X_{31}$ and $[23|12]$ are regular modulo $P'$ is already worked out in \S5.2(d). Similarly, we check that $X_{23}A'$ and $X_{12}A'$ are prime ideals of $A'$ which do not contain $X_{12}X_{23}- X_{31}$, and then \S5.2(3) implies that $X_{23}$ and $X_{12}$ are regular modulo $P'$. Therefore we conclude that $P$ is a prime ideal, completing the case $w= (321,312)$. 

The cases $(231,321)$, $(321,231)$, and $(312,321)$ follow via $\tau$, $S$, and $\tau S$, respectively, taking account of (E4.3b) in the case $(321,231)$.

\lowunder{$w= (231,231)$}. We must show that the ideals
\begin{align*}
P &:= Q_w+ \langle D_q-\alpha,\, [13|23]-\beta X_{21} \rangle  \\
M &:= Q_w+ \langle D_q-\alpha,\, [13|23]-\beta X_{21},\, [12|13]-\gamma X_{32} \rangle
\end{align*}
are prime. Without loss of generality, $\alpha=\beta=\gamma=1$. It suffices to show that the generators of $E_w$, namely
$$X_{21}\,,\; X_{32}\,,\; [13|23]\,,\; X_{13}\,,$$
are regular modulo $P$ and $M$. We shall work with the images of $P$ and $M$ in the domain $A' := A/ \bigl( Q_w+ \langle D_q- 1\rangle \bigr)$, which we denote $P'$ and $M'$.

Via \S5.2(1)(2), we see that $X_{21}A'$ and $X_{32}A'$ are prime ideals of $A'$ which do not contain $[13|23]- X_{21}$. Hence, \S5.2(3) implies that $X_{21}$ and $X_{32}$ are regular modulo $P'$. Since $[13|23]$ is congruent to $X_{21}$ modulo $P'$, it follows that $[13|23]$ is regular modulo $P'$. 

For the regularity of $X_{13}$ modulo $P'$, we show that $[13|23]- X_{21}$ is regular modulo $X_{13}A'$. Observe that $A'/X_{13}A' \cong A''/[12|23]A''$ where 
$$A'' := A/ \bigl( Q_{231,312}+ \langle D_q-1\rangle \bigr).$$
Hence, it will be enough to show that $[13|23]- X_{21}$ is regular modulo $[12|23]A''= X_{12}X_{23}A''$. For that, regularity modulo both $X_{12}A''$ and $X_{23}A''$ will suffice. Via \S5.2(1), we see that $X_{12}A''$ and $X_{23}A''$ are prime ideals of $A''$. Since
$$Q_{231,312}+ \langle D_q-1,\, X_{12},\, X_{23}\rangle= Q_{231,123}+ \langle D_q-1\rangle \subsetneq Q_{132,123}+ \langle D_q-1\rangle,$$
we find that $[13|23]- X_{21} \notin X_{12}A''+ X_{23}A''$. Thus, $[13|23]- X_{21}$ is regular modulo both $X_{12}A''$ and $X_{23}A''$, as desired.

Therefore $X_{13}$ is regular modulo $P'$, concluding the proof that $P'$ is prime.
 
Inspecting $Q_{132,123}$, which contains $Q_{231,231}$ as well as $X_{21}$ and $[12|13]$ but not $X_{32}$ (by \S5.2(2)), we see that $[12|13]-X_{32}$ is not in $X_{21}A'$. We have already observed that the latter ideal is prime. Hence, $[12|13]-X_{32}$ is regular modulo $X_{21}A'$. From the case $(132,231)$ done in (c) above, we know that
$$Q_{132,231}+ \langle D_q-1,\, X_{11}X_{23}-X_{32}\rangle= Q_{132,231}+ \langle D_q-1,\, [12|13]-X_{32}\rangle$$
is a prime ideal of $A$, and thus $X_{21}A'+ ([12|13]-X_{32})A'$ is prime. Inspecting $Q_{123,213}$, which contains $Q_{231,231}$ as well as $X_{21}$, $X_{32}$, and $[12|13]-X_{32}$ but not $[12|13]$, we see that $[13|23]- X_{21}\notin X_{21}A'+ ([12|13]-X_{32})A'$. Hence, $[13|23]- X_{21}$ is regular modulo $X_{21}A'+ ([12|13]-X_{32})A'$. We now conclude from \S5.2(4) that $X_{21}$ is regular modulo $M'$.

A symmetric argument shows that $X_{32}$ is regular modulo $M'$. Since $[13|23|$ and $[12|13]$ are congruent to $X_{21}$ and $X_{32}$ modulo $M'$, it follows that $[13|23|$ and $[12|13]$ are regular modulo $M'$. In view of (E4.3b), $X_{13}$ is congruent to $[12|13][13|23]$ modulo $M'$, and thus it is regular modulo $M'$. We now conclude that $M'$ is prime, concluding the case $(231,231)$.

The case $(312,312)$ follows via $\tau$.

\lowunder{$w= (321,321)$}. Here $Q_w=0$, and we must show that the ideals
\begin{align*}
P &:= \langle D_q-\alpha,\, [23|12]-\beta X_{13} \rangle  \\
M &:= Q_w+ \langle D_q-\alpha,\,  [23|12]-\beta X_{13} ,\, [12|23]-\gamma X_{31} \rangle
\end{align*}
are prime. Without loss of generality, $\alpha=\beta=\gamma=1$. It suffices to show that the generators of $E_w$, namely
$$X_{31}\,,\; [23|12]\,,\; [12|23]\,,\; X_{13}\,,$$
are regular modulo $P$ and $M$. We shall work with $P' := P/ \langle D_q-1\rangle$ and $M' := M/\langle D_q-1\rangle$ in the domain $A' := A/\langle D_q-1\rangle$.

Via \S5.2(1), we see that $X_{31}A'$, $[12|23]A'$, and $X_{13}A'$ are prime. Since
$$\langle D_q-1,\, X_{31}\rangle \subsetneq Q_{123,321}+ \langle D_q-1\rangle \subsetneq Q_{123,312}+ \langle D_q-1\rangle,$$
we see by \S5.2(2) that $[23|12]-X_{13} \notin X_{31}A'$. Similarly, this element is not in either $[12|23]A'$ or $X_{13}A'$, as we see by inspecting $Q_{321,123} \subsetneq Q_{312,123}$ and $Q_{321,312} \subsetneq Q_{312,312}$. Hence, $[23|12]-X_{13}$ is regular modulo each of $X_{31}A'$, $[12|23]A'$, and $X_{13}A'$. By \S5.2(3), $X_{31}$, $[12|23]$, and $X_{13}$ are all regular modulo $P'$. Moreover, $[23|12]$ is congruent to $X_{13}$ modulo $P'$, and so it is regular modulo $P'$. Therefore $P'$ is a prime ideal of $A'$.

As just checked, $[23|12]-X_{13}$ is regular modulo $X_{13}A'$. By \S5.2(1),
$$\langle D_q-1,\, X_{13},\, [23|12]-X_{13} \rangle = Q_{312,312}+ \langle D_q-1 \rangle$$
is prime in $A$, and so $X_{13}A'+ ([23|12]-X_{13})A'$ is prime in $A'$. This ideal does not contain $[13|23]-X_{31}$, as we see by inspecting $Q_{123,312} \subsetneq Q_{123,132}$ and $Q_{123,312} \subsetneq Q_{123,213}$, and so $[13|23]-X_{31}$ is regular modulo $X_{13}A'+ ([23|12]-X_{13})A'$. Thus, \S5.2(4) implies that $X_{13}$ is regular modulo $M'$. Applying $\tau$, which induces an automorphism of $A'$ stabilizing $M'$, we find that $X_{31}$ is regular modulo $M'$. Since $[23|12]$ and $12|23]$ are congruent to $X_{13}$ and $X_{31}$ modulo $M'$, it follows that $[23|12]$ and $12|23]$ are regular modulo $M'$. Therefore $M'$ is a prime ideal of $A'$, concluding the last case of the lemma.
\end{proof}

\begin{corollary} \label{cor5.4}
Let $w\in S_3\times S_3$, and let $a_1,\dots,a_d$ be the elements of $A/Q_w$ listed in position $w$ of Figure {\rm\ref{polyprimegens}} {\rm(}for some choices of $\alpha,\beta,\gamma \in \kx${\rm)}. Then $a_1,\dots,a_d$ is a normal regular sequence in $A/Q_w$.
\qed\end{corollary}

We can now establish our  main theorem.

\begin{theorem}  \label{mainthm} Let $A= \OqGLth$, with $k$ algebraically closed and $q$ not a root of unity.

{\rm (a)} Let $w\in S_3\times S_3$, and let $a_1,\dots,a_d$ be the elements listed in position $w$ of Figure {\rm\ref{polyprimegens}} {\rm(}for some choices of $\alpha,\beta,\gamma \in \kx${\rm)}, now viewed as elements of $A$. Then $Q_w+ \langle a_1,\dots,a_d\rangle$ is a primitive ideal of $A$.

{\rm (b)} The ideals described in {\rm(a)} constitute all the primitive ideals of $A$.
\end{theorem}

\begin{proof} (a) Let $z_1,\dots,z_d$ be the elements of $A_w$ listed in position $w$ of Figure \ref{indetsZAw}, and write $\alpha_1= \alpha$, $\alpha_2= \beta$, \dots. 
Set $P= Q_w+ \langle a_1,\dots,a_d\rangle$, and observe that
$$(P/Q_w)A_w = (z_1-\alpha_1)A_w +\cdots+ (z_d-\alpha_d)A_w \,,$$
which is a prime ideal of $A_w$ by \S\ref{subsec5.1}(1). 
By Lemma 5.3, $P/Q_w$ is a prime ideal of $A/Q_w$, from which we conclude that $P/Q_w= (P/Q_w)A_w \cap (A/Q_w)$. Consequently, \S\ref{subsec5.1}(2) implies that  $P$ is primitive.

(b) If $P$ is a primitive ideal of $A$, then $P\in \prim_w A$ for some $w\in S_3\times S_3$. In view of \S\ref{subsec5.1}(3),
$$P/Q_w= \bigl( (z_1-\alpha_1)A_w +\cdots+ (z_l-\alpha_l)A_w \bigr) \cap (A/Q_w)$$
for some $\alpha_i\in\kx$, where $z_1,\dots,z_d$ are the elements of $A_w$ listed in position $w$ of Figure \ref{indetsZAw}. Set $\alpha=\alpha_1$, $\beta=\alpha_2$, \dots, and let $a_1,\dots,a_d$ be the elements listed in position $w$ of Figure {\rm\ref{polyprimegens}}. As shown in the proof of (a) above, $P/Q_w$ equals the ideal of $A/Q_w$ generated by the cosets of $a_1,\dots,a_d$, and therefore $P= Q_w+ \langle a_1,\dots,a_d\rangle$, as desired.
\end{proof}

Let $\pi: \OqGLth \rightarrow \OqSLth$ denote the canonical quotient map. Since the primitive ideals of $\OqSLth$ are precisely the ideals of the form $\pi(P)$ where $P$ is a primitive ideal of $\OqGLth$ containing $D_q-1$, generators for the primitive ideals of $\OqSLth$ can be immediately obtained from Theorem \ref{mainthm}, as follows.

\begin{sidewaysfigure} 
$$\begin{matrix}
 &&&321  &231  &312  &132  &213  &123 \\  \\
321  &&&[23|12]-\beta X_{13}  &[23|12]-\beta X_{13}  &X_{12}X_{23}-\beta X_{31}  &0  &0  &X_{22}[23|12]-\beta X_{31} \\
 &&&[12|23]-\gamma X_{31}  \\  \\
231  &&&X_{21}X_{32}-\beta X_{13}  &[13|23]-\beta X_{21}  &0  &X_{11}X_{23}-\beta X_{32}  &X_{12}X_{33}-\beta X_{21}  &0  \\
 &&&&[12|13]-\gamma X_{32}  \\  \\
312  &&&[12|23]-\beta X_{31}  &0  &[23|13]-\beta X_{12}  &X_{11}X_{32}-\beta X_{23}  &X_{21}X_{33}-\beta X_{12}  &0  \\
 &&&&&[13|12]-\gamma X_{23}  \\  \\
 &&&&&&X_{11}-\alpha  &&X_{11}-\alpha  \\
132  &&&0  &X_{11}X_{23}-\beta X_{32}  &X_{11}X_{32}-\beta X_{23}  &[23|23]-\alpha^{-1}  &0  &X_{22}X_{33}-\alpha^{-1}   \\
 &&&&&&X_{23}-\gamma X_{32}  \\  \\
 &&&&&&&X_{33}-\alpha  &X_{11}X_{22}-\alpha  \\
213  &&&0  &X_{12}X_{33}-\beta X_{21}  &X_{21}X_{33}-\beta X_{12}  &0  &[12|12]-\alpha^{-1}   &X_{33}-\alpha^{-1}   \\
 &&&&&&&X_{12}-\gamma X_{21}  \\  \\
 &&&&&&X_{11}-\alpha  &X_{11}X_{22}-\alpha  &X_{11}-\alpha  \\
123  &&&X_{22}[12|23]-\beta X_{13}  &0  &0  &X_{22}X_{33}-\alpha^{-1}   &X_{33}-\alpha^{-1}   &X_{22}-\beta   \\
 &&&&&&&&X_{33}-\alpha^{-1}\beta^{-1}
 \end{matrix}$$
\caption {Generators for primitive ideals in factor algebras $\OqSLth/Q_w$}   \label{OqSL3privgens}
\end{sidewaysfigure}

\begin{corollary}  \label{maincor} Let $B= \OqSLth$, with $k$ algebraically closed and $q$ not a root of unity.

{\rm (a)} Let $w\in S_3\times S_3$, and let $a_1,\dots,a_d$ be the elements listed in position $w$ of Figure {\rm\ref{OqSL3privgens}} {\rm(}for some choices of $\alpha,\beta,\gamma \in \kx${\rm)}, now viewed as elements of $B$. Moreover, view $Q_w$ as an ideal of $B$ {\rm(}as defined in {\rm\S\ref{subsec2.2})}. Then $Q_w+ \langle a_1,\dots,a_d\rangle$ is a primitive ideal of $B$.

{\rm (b)} The ideals described in {\rm(a)} constitute all the primitive ideals of $B$.
\qed\end{corollary}


\section{General consequences} \label{consequences}

\begin{theorem} \label{thmXX} All primitive factor algebras of $\OqGLth$ and $\OqSLth$ are Auslander-Gorenstein and GK-Cohen-Macaulay {\rm(}assuming $k$ algebraically closed and $q$ not a root of unity{\rm)}.
\end{theorem}

\begin{proof} Let $P$ be an arbitrary primitive ideal of $\OqGLth$, and let $w\in S_3\times S_3$ such that $P\in \prim_w \OqGLth$. By \S\ref{subsec2.3}(4), Theorem \ref{mainthm}, and Corollary \ref{cor5.4}, $Q_w$ and $P/Q_w$ both have polynormal regular sequences of generators. It follows that $P$ has a polynormal regular sequence of generators. Since $\OqGLth$ is Auslander-regular and GK-Cohen-Macaulay (e.g., \cite[Proposition I.9.12]{BG}), we conclude from Theorem \ref{modnormal} that $\OqGLth/P$ is Auslander-Gorenstein and GK-Cohen-Macaulay.

The remaining statement is immediate from the fact that every primitive factor algebra of $\OqSLth$ is also a primitive factor algebra of $\OqGLth$.
\end{proof}

By inspection, each primitive ideal of $\OqGLth$ is contained in one of the primitive ideals in $\prim_{123,123} \OqGLth$. Since maximal ideals are primitive, we find that the only maximal ideals of $\OqGLth$ are those in $\prim_{123,123} \OqGLth$, and similarly in $\OqSLth$. This establishes the following result, answering two cases of a question raised in \cite[Introduction]{GZ}.

\begin{theorem} \label{thmYY} Every maximal ideal of $\OqGLth$ and $\OqSLth$ has codimension $1$ {\rm(}assuming $k$ algebraically closed and $q$ not a root of unity{\rm)}.
\end{theorem}


\section{Appendix. Homological conditions}

\begin{definition} A noetherian ring $R$ is \emph{Auslander-Gorenstein} provided
\begin{enumerate}
\item The  modules $R_R$ and $_RR$ both have finite injective dimension;
\item $R$ satisfies the \emph{Auslander condition}: $\calExt^i_R(N,R)=0$ for all $R$-submodules $N$ of $\calExt^j_R(M,R)$ whenever $0\le i<j$ and $M$ is a finitely generated (right or left) $R$-module.
\end{enumerate}
If condition (1) is strengthened to `$\gldim R<\infty$', then $R$ is \emph{Auslander-regular}.

The \emph{grade} (or \emph{$j$-number}) of a finitely generated $R$-module $M$ is
$$j(M)= j_R(M) := \inf \{ j\ge0 \mid \calExt^j_R(M,R) \ne 0 \}.$$

Now assume that $R$ is an affine $k$-algebra. Then $R$ is \emph{GK-Cohen-Macaulay} provided $\GKdim(R)<\infty$ and
$$j(M)+ \GKdim(M)= \GKdim(R)$$
for every nonzero finitely generated (right or left) $R$-module $M$. 
\end{definition}

\begin{theorem} \label{modnormal}
Let $R$ be a noetherian ring, and let $\Omega\in R$ be a regular normal element.

{\rm (a)} If $R$ is Auslander-Gorenstein, then so is $R/\Omega R$.

{\rm (b)} Assume that $R$ is an affine $k$-algebra. If $R$ is GK-Cohen-Macaulay, then so is $R/\Omega R$.
\end{theorem}

\begin{proof} (a) \cite[\S3.4, Remark 3]{Lev}.

(b) By \cite[\S3.4, Remark 3]{Lev}, $j_{R/\Omega R}(M) = j_R(M)-1$
for any nonzero finitely generated $(R/\Omega R)$-module $M$. Since $R$ is GK-Cohen-Macaulay, we get
\begin{equation}
j_{R/\Omega R}(M)+1+\GKdim(M)= \GKdim(R).  \tag{E7.2}
\end{equation}
The case $M=R/\Omega R$ of (E7.2) implies that $1+\GKdim(R/\Omega R)= \GKdim(R)$, and hence (E7.2) can be rewritten as
$$j_{R/\Omega R}(M)+\GKdim(M)= \GKdim(R/\Omega R).$$
This shows that $R/\Omega R$ is GK-Cohen-Macaulay.
\end{proof} 


\end{document}